\documentclass{article}
\usepackage{amsmath}
\usepackage{amssymb}
\usepackage{mathtools}
\usepackage{dsfont}
\usepackage[mathscr]{eucal}
\usepackage{mathrsfs}
\usepackage{stmaryrd}
\usepackage{centernot}

\usepackage[utf8]{inputenc}
\usepackage[T1]{fontenc}
\usepackage{csquotes}
\usepackage{microtype}
\usepackage[margin=1in]{geometry}

\usepackage{tikz-cd}
\tikzset{
  symbol/.style={
    draw=none,
    every to/.append style={
      edge node={node [sloped, transform shape, auto=false]{\scalebox{2}[1]{$#1$}}}}
  }
}

\usepackage{amsthm}
\newtheorem{theoremcounter}{DONTUSETHIS}[section]

\usepackage{enumitem}
\newlist{parts}{enumerate}{1}
\setlist[parts]{label=(\alph*),ref=\thetheoremcounter~(\alph*)}
\newlist{tfae}{enumerate}{1}
\setlist[tfae]{label=(\roman*)}

\usepackage{hyperref}
\usepackage[capitalise,compress]{cleveref-forward}
\crefname{tfae}{condition}{conditions}
\Crefname{tfae}{Condition}{Conditions}
\newcommand{\createtheoremtype}[3]{%
  \newtheorem{#1}[theoremcounter]{\MakeUppercase #2}%
  \crefname{#1}{#2}{#3}%
  \Crefname{#1}{\MakeUppercase #2}{\MakeUppercase #3}%
  \AddToHook{env/#1/begin}{\crefalias{theoremcounter}{#1}}
  \AddToHook{env/#1/begin}{\crefalias{partsi}{#1}}
}

\createtheoremtype{lemma}{Lemma}{Lemmas}
\createtheoremtype{proposition}{Proposition}{Propositions}
\createtheoremtype{theorem}{Theorem}{Theorems}
\createtheoremtype{corollary}{Corollary}{Corollaries}
\createtheoremtype{conjecture}{Conjecture}{Conjectures}
\theoremstyle{definition}
\createtheoremtype{definition}{Definition}{Definitions}
\theoremstyle{remark}
\createtheoremtype{fact}{Fact}{Facts}
\createtheoremtype{remark}{Remark}{Remarks}
\createtheoremtype{example}{Example}{Examples}

\newtheoremstyle{sublemma}{\topsep}{\topsep}{\itshape}{}{\bf}{.}{.5em}{}
\theoremstyle{sublemma}
\newtheorem{sublemma}{Claim}[theoremcounter]
\crefname{sublemma}{Claim}{Claims}
\Crefname{sublemma}{Claim}{Claims}

\newenvironment{proof*}{\vspace{-.6em}\begin{proof}}{\end{proof}\vspace{-.6em}}
\theoremstyle{definition}

\def\labelfull#1#2{\label[#1]{#1}
  \crefformat{#1}{##2#2##3}\Crefformat{#1}{##2#2##3}}

\def\labelsthm#1#2{\label[#1]{#1}
  \crefformat{#1}{##2#2##3's theorem}\Crefformat{#1}{##2#2##3's theorem}}

\newcommand{\N}{\mathbb N}
\newcommand{\Z}{\mathbb Z}
\newcommand{\Q}{\mathbb Q}
\newcommand{\R}{\mathbb R}
\newcommand{\C}{\mathbb C}

\let\strokeL\L
\DeclareUnicodeCharacter{0141}{\strokeL}
\renewcommand{\L}{\mathbb L}

\newcommand{\msf}{\mathsf}
\newcommand{\mc}{\mathcal}

\newcommand{\msc}{\mathscr}

\renewcommand{\phi}{\varphi}
\renewcommand{\theta}{\vartheta}

\newcommand{\xto}{\xrightarrow}

\newcommand{\theo}[1]{\mathsf{#1}}

\newcommand{\pare}[1]{\left(#1\right)}
\newcommand{\set}[1]{\left\{#1\right\}}
\newcommand{\setb}[2]{\set{#1\,\middle|\,#2}}

\newcommand{\inter}[1]{\left[#1\right]}
\newcommand{\lointer}[1]{\left(#1\right]}
\newcommand{\rointer}[1]{\left[#1\right)}
\newcommand{\ointer}[1]{\left(#1\right)}

\newcommand{\abs}[1]{\left|#1\right|}

\DeclareMathOperator{\id}{id}

\DeclareMathOperator{\dom}{dom}

\DeclareMathOperator{\im}{im}

\newcommand{\tsingl}{\{\star\}}

\newcommand{\dual}[1]{#1^{\mathrm{op}}}

\DeclareMathOperator{\diam}{diam}
\DeclareMathOperator{\rk}{rk}

\usepackage[style=alphabetic,maxnames=99,maxbibnames=99,maxcitenames=99,isbn=false,doi=false,url=false]{biblatex}
\newcommand{\Th}{\msf{Th}}
\newcommand{\fTh}{\msf{fTh}}
\newcommand{\pTh}{\msf{pTh}}
\newcommand{\uTh}{\msf{uTh}}
\newcommand{\cTh}{\msf{cTh}}
\newcommand{\fcTh}{\msf{fcTh}}

\newcommand{\isAtom}{\msf{isAtom}}
\newcommand{\asub}{\msf{asub}}
\newcommand{\elem}{\msf{elem}}
\newcommand{\refin}{\msf{refin}}
\newcommand{\urefin}{\msf{urefin}}

\newcommand{\exten}{\msf{exten}}
\newcommand{\recon}{\msf{recon}}
\newcommand{\ucsub}{\msf{ucsub}}
\newcommand{\cant}{\msf{cant}}
\newcommand{\ucant}{\msf{ucant}}
\newcommand{\urefincant}{\msf{urefincant}}

\newcommand{\pows}[1]{2^{#1}}

\begin{document}
\title{The Borel monadic theory of order is decidable}
\author{Sven Manthe}
\date{}
\maketitle
\begin{abstract}
  The monadic theory of $(\R,\le)$ with quantification restricted to Borel sets is decidable. The Boolean combinations of $F_\sigma$ sets form an elementary substructure of the Borel sets. Under determinacy hypotheses, the proof extends to larger classes of sets.
\end{abstract}

\section{Introduction}
\cite{treeautomata} established decidability of the monadic second-order theory of two successors, S2S. Formulas in its language can be described by automata operating on infinite trees. S2S is a highly expressive decidable theory, used to establish decidability of other theories. It is a key result in the research investigating decidability of monadic second-order theories. For uncountable structures, decidability often necessitates restricting the range of the monadic quantifier to tamer classes of sets. For instance, S2S interprets the monadic theory of $(\R,\le)$ with quantification restricted to $F_\sigma$ sets, yielding decidability of the latter theory \cite[Theorem 2.9]{treeautomata}. Here, we answer affirmatively the open question \cites[Conjecture 7B]{shelahord}[$\alpha$(5), p. 95]{shelahnotesb}[\nopp 2.22]{shelahquestions} whether decidability persists when $F_\sigma$ sets are replaced by Borel sets.
\begin{theorem}\label{corborel}
  The monadic theory of $(\R,\le)$ with quantification restricted to Borel sets is decidable. The Boolean combinations of $F_\sigma$ sets form an elementary substructure of the Borel sets.
\end{theorem}
Without any restriction on the quantifiers, the theory is undecidable. Undecidability was initially shown under the continuum hypothesis \cite[Section 7]{shelahord} and subsequently in $\theo{ZFC}$ \cite{monadicundecid}. In fact, this theory is even more expressive than most undecidable theories. First-order arithmetic can be reduced to it. More generally, third-order arithmetic in a forcing extension obtained by adding a Cohen real to the ambient model of set theory can be reduced to the unrestricted monadic theory (first shown under the continuum hypothesis \cite{nextworld} and then in $\theo{ZFC}$ \cite{shelahnotesa}). A survey of the earlier developments is \cite{monadicsurvey}.\par
The monadic theory of topology is obtained by replacing the order relation by a unary predicate for open sets (equivalently, by the closure operator). For the reals or the Cantor set, the monadic theory of order readily interprets the monadic theory of topology. For the reals, the monadic theory of topology also defines the relation $x<y<z\lor z<y<x$ using connectedness, whence its decidability is equivalent to decidability of the theory of order. Thus, it is interesting to note that for many spaces, such as $\R^n$ for $n\ge2$, the monadic theory of topology becomes undecidable even when quantification is restricted to simple classes of sets \cites[Corollary 1]{undecidableRn}{undecidableRn2}. Generalizations of the undecidability result from the reals to other spaces have been studied \cite{monordtopi,monordtopii,modchaini,modchainii}.\par
Before establishing undecidability of the unrestricted monadic theory, \cite{shelahord} also introduced a novel technique for showing decidability of monadic theories of order. Here, we extend Shelah's methods using the Baire property in the Borel setting.\par
In the monadic language of order, for fixed $m,n\in\N$, there are up to equivalence only finitely many formulas in $m$ variables with $n$ quantifiers. Thus, decidability amounts to computability of the subset of this finite set given by true formulas, uniformly in $m,n$. Using lexicographic sums and Ramsey theoretic methods, Shelah reduced this problem to computability for formulas only containing quantifiers over sufficiently uniform sets and quantifiers over Cantor sets. A set is sufficiently uniform if it satisfies the same formulas with $n$ quantifiers on every open interval. In the unrestricted monadic theory, these quantifiers still suffice to encode arithmetic. This encoding requires a (uniform) parameter defined by transfinite recursion.\par
We show that the Borel monadic theory defines the class of meager sets. Using the definability, the Baire property and specific generalized Wadge games, we describe uniform formulas in variables $X_1,\dots,X_n$ combinatorially. Such formulas merely specify the existence or absence of possibly nested Cantor sets meeting certain $X_i$ while avoiding others. A more detailed outline is \cref{secsketch}.\par
The notion of a \textbf{sufficiently stable} Boolean subalgebra $\mc B$ of $\pows\R$ is defined in \cref{secconseq} such that both the Borel sets and the Boolean algebra generated by the $F_\sigma$ sets are stably Baire. Thus, our main result is a corollary of the following theorem.
\begin{theorem}[ZF+Dependent choices]\label{mainres}
  Let $\mc B,\mc C$ be sufficiently stable Boolean algebras with $\mc B\subseteq\mc C$ such that every subset of $2^\N$ in $\mc C$ is determined.
  \begin{parts}
  \item The monadic theory of $(\R,{\le})$ with quantification restricted to $\mc B$ is decidable.
  \item The embedding $\mc B\to\mc C$ is elementary.
  \end{parts}
\end{theorem}
\Cref{secconseq} presents further corollaries of this theorem.\par
The monadic theory of the reals expresses being homeomorphic to the Cantor set as being compact, nonempty, without isolated points and with empty interior. Similarly, the Baire space can be characterized as the complement of a countable dense subset. Hence, our results remain valid if the reals are substituted by one of these spaces (see \cref{equivcomp}, and the related \cref{decidpolish}).\par
Another monadic theory whose decidability remains open is that of $2^{\le\omega}$ with the lexicographic ordering and the prefix relation. This theory interprets both the monadic theory of $(\R,\le)$ and S2S. Our methods seem insufficient for addressing this theory, as explained in \cref{secbaire}.\par
Before continuing with \cref{secsketch,secconseq,secbaire}, we describe the structure of the article. \Cref{sectoppre} provides prerequisites in topology, order theory and Ramsey theory. \Cref{sectopo} presents point-set topological lemmas used in \cref{secdecproofa} and uses them to show the definability of meagerness. In \cref{secrepshe}, we present an adjusted version of results of \cite{shelahord}. We restrict quantification and simplify by restricting to complete total orders, for example, using Cantor sets instead of convex equivalence relations. Finally, \cref{secdecproofa} combines the previous parts to describe a fragment of the theory explicitly by a small amount of data for uniform types, relying only on the Baire property, and \cref{secdecproofb} employs generalized Wadge games to reduce all uniform types to this fragment.

\subsection{Proof outline}\label{secsketch}
In the following, we sketch both the technique of \cite{shelahord} and its extension to prove \cref{mainres}.\par
For fixed $m\in\N$ and $k\in\N^n$ define a set $\fTh^-_k(m)$ of monadic formulas in $m$ free variables by recursion on $n$. Let $\fTh^-_{()}(m)$ be the set of atomic formulas. For $n>0$ let $\fTh^-_k(m)$ be the set of formulas $\exists X_1\dots\exists X_{k_n}\phi$ with $\phi$ ranging over the Boolean combinations of elements of $\fTh^-_{(k_1,\dots,k_{n-1})}(m+k_n)$. Since the language is relational, each $\fTh^-_k(m)$ is finite. For a tuple $X=(X_1,\dots,X_m)$ of Borel sets, set $\Th_k(\R,X)=\setb{\phi\in\fTh^-_k(m)}{\R\models\phi(X)}$. Our aim is to compute $\Th_k(\R)$ (setting $m=0$) as $k$ varies. (Actually, we use a slightly different definition of $\Th_k$, considering first-order quantifiers separately to eliminate more quantifiers.)\par
Generalizing to arbitrary total orders, consider a lexicographic sum $\sum_{i\in I}X_i$ of total orders. For finite index sets $I$, $\Th_k(\sum_iX_i)$ is determined by and computable from the $\Th_k(X_i)$. This generalizes to arbitrary lexicographic sums if the input incorporates the monadic theory of $I$ in an appropriate sense (\cref{monadicsum}). For example, if $\R$ is the disjoint union of half-open intervals $\lointer{a_i,a_{i+1}}$ with a common type $\Th_k(\lointer{a_i,a_{i+1}},X)=t$ for all $i$, then $\Th_k(\R,X)$ can be computed from $t$, and if $C\subseteq\R$ is a Cantor set whose complementary intervals share a common type $t$, then $\Th_k(\R,X)$ can be computed from $\Th_k(C,X)$ and $t$.\par
Call a tuple $X$ of Borel sets $k$-uniform if any pair of open intervals $\ointer{a,b},\ointer{c,d}$ satisfies $\Th_k(\ointer{a,b},X)=\Th_k(\ointer{c,d},X)$. For example, for $k$ large enough a single countable set is $k$-uniform if and only if it is either empty or dense. Despite the restrictiveness of this property, by a Ramsey theoretic argument every tuple is $k$-uniform on some open interval. Thus, the set of points with a $k$-uniform neighborhood is open dense. In fact, its complement is empty or a Cantor set $C\subseteq\R$. In the latter case, by the aforementioned computation for lexicographic sums, $\Th_k(\R,X)$ is essentially determined by $\Th_k(C,X)$.\par
Employing this line of reasoning, all monadic sentences can be transformed into sentences containing only the following two kinds of restricted quantifiers (\cref{shelahres}):
\begin{itemize}
\item $\exists Y\phi(X_1,\dots,X_m,Y)$ quantifying over all $Y$ such that $(X_1,\dots,X_m,Y)$ is still uniform,
\item $\exists C\phi(X_1,\dots,X_m,C)$ quantifying over all Cantor sets $C$ for which $(X_1,\dots,X_m)$ is also uniform on $C$.
\end{itemize}
The first kind of quantifier can often be eliminated or reduced to the second kind. The second kind can be more problematic.\par
The techniques described so far also apply to the unrestricted monadic theory and originate in \cite{shelahord}. However, in the unrestricted theory, certain sets encode Boolean-valued first-order structures. In contrast, in the Borel setting, Baire category arguments control Cantor sets.\par
A prototypical instance of the remaining question is: Given a $k$-uniform tuple $(X_1,\dots,X_m)$, is there a Cantor set $C$ such that $X_1,\dots,X_\ell$ are dense in $C$, while $X_{\ell+1},\dots,X_m$ are disjoint from $C$? We do not get full quantifier elimination since the answer depends not only on the monadic quantifier-free type of $X_1,\dots,X_m$. For example, consider $m=3,\ell=2$. If $X_1,X_2$ are disjoint meager $F_\sigma$ sets and $X_3=\R\setminus(X_1\cup X_2)$, then no such Cantor set exists, but if $X_1$ is comeager instead, then such a set does exist.\par
Based on these ideas, we define meagerness in the Borel monadic theory (\cref{secdefinemeager}). By the Baire property, a set is either meager or comeager on some open interval. Hence, a $k$-uniform set is either meager or comeager everywhere.\par
After Shelah's reduction, the decision procedure consists of two parts.\par
First, we define a coarse version $\cTh_n$ of $\Th_k$ restricted to quantifiers of the second kind, where $n$ is the number of quantifier alternations. (The actual definition proceeds in an expanded language.) We compute the set of satisfiable coarse types of uniform tuples (\cref{coarsecomput}). In other words, we decide sentences with an existential block of quantifiers of the first kind followed by a formula containing only quantifiers of the second kind.\par
Each meager subset of the reals is included in a countable union of Cantor sets. Consider a type $\cTh_n(\R,X_1,\dots,X_n)$ and assume without loss of generality that all $X_i$ are meager. By a Baire category argument, $\cTh_n(\R,X_1,\dots,X_n)$ is determined by the types $\cTh_n(C_m)$ of the countably many Cantor sets $C_m$ with $\bigcup_mC_m\supseteq\bigcup_iX_i$. In other words, these Cantor sets do not interact. We finally show that every coarse type $\cTh_n(\R,P)$ arises from a finite iteration of a construction termed ``uniform sum'' (\cref{satisconstr}).\par
Intuitively, the uniform sum of Cantor sets $C_1,\dots,C_n$ (equipped with tuples of subsets) is the simplest way to construct uniform meager unions of copies of the $C_i$. For instance, consider $n=1$: First embed $C_1$ into $\R$; then embed $C_1$ into each complementary interval of the previous embedding, obtaining a new Cantor set that is a countable union of copies of $C_1$; finally iterate embedding $C_1$ into the complementary intervals thus obtained.\par
The second part of the decision procedure computes $\Th_k$ from $\cTh_n$ for appropriate $n$. If sufficiently complicated Cantor subsets exist, this step is straightforward. Otherwise, the reals can be decomposed into Boolean combinations of few $F_\sigma$ sets, which can be analyzed independently. This dichotomy follows from the determinacy of a ``separation game'' (\cref{secsepgame}), a generalized Wadge game.

\subsection{Consequences}\label{secconseq}
\begin{definition}\label{adequateall}
  A family $(\mc B_X)_X$ of Boolean subalgebras of $\pows X$ for $X$ ranging over the $G_\delta$ subsets of $\R$ is \textbf{sufficiently stable} if there are $\mc A_X\subseteq\pows X$ such that
  \begin{parts}
  \item $\mc A_X$ generates $\mc B_X$ as a Boolean algebra;
  \item if $f\colon X\to Y$ is a homeomorphism and $A\in\mc A_X$, then $f(A)\in\mc A_Y$;
  \item if $A\subseteq X$ is $G_\delta$ and $B\in\mc A_X$, then $A\cap B\in\mc A_A$;
  \item if $\R$ is partitioned into a $G_\delta$ subset $Y$ and countably many locally closed subsets $X_i$, then $A\in\mc B_\R$ if $A\cap Y\in\mc A_Y$ and $A\cap X_i\in\mc A_{X_i}$ for all $i$.\label{adequatepart}
  \end{parts}
  A Boolean subalgebra of $\pows\R$ is sufficiently stable if it is a component of a sufficiently stable family of Boolean algebras.
\end{definition}
The Borel sets are sufficiently stable with $\mc A$ comprising all Borel sets. The Boolean combinations of $F_\sigma$ sets are sufficiently stable with $\mc A$ consisting of the $G_\delta$ sets. Similarly, the Boolean combinations of $\mathbf\Sigma^0_\alpha$ sets for $\alpha\ge2$ or the $\mathbf\Delta^0_\alpha$ sets for $\alpha\ge3$ form sufficiently stable Boolean algebras, as do the Boolean or $\sigma$-combinations of $\mathbf\Sigma^1_n$ sets for $n\ge1$ or the $\mathbf\Delta^1_n$ sets, or the projective sets.\par
Thus, assuming the corresponding determinacy, we obtain analogues of \cref{corborel} by replacing either the Borel sets or the Boolean combinations of $F_\sigma$ sets by any of these algebras. For instance:
\begin{corollary}[Projective determinacy]
  The monadic theory of $(\R,\le)$ with quantification restricted to projective sets is decidable. The Boolean combinations of $F_\sigma$ sets form an elementary substructure of the projective sets.
\end{corollary}
Nevertheless, undecidability of the theory with quantification over projective sets is also consistent with $\theo{ZFC}$. More precisely, undecidability holds if a projective well-ordering of the reals exists. The undecidability proof applies as noted in \cite[p. 410]{shelahord}. This holds without the continuum hypothesis using the proofs in \cite{monadicundecid} or \cite{shelahnotesa}. A detailed explanation under the continuum hypothesis is \cite[Theorem 3.1]{msoplusu}. In fact, the axiom of constructibility yields a $\Delta^1_2$-well-ordering.\par
Similarly, we deduce that the axiom of choice is required for undecidability of the full theory.
\begin{corollary}[ZF+Dependent choices+Determinacy]
  The monadic theory of $(\R,\le)$ is decidable. The Boolean combinations of $F_\sigma$ sets form an elementary substructure of all sets.
\end{corollary}

\subsection{Definability in the Borel monadic theory of order}
Many topological notions are definable in the monadic language of topology (and thus of order), for example, dense, connected and perfect sets, or sets with exactly $n$ isolated points for fixed $n\in\N$. The Borel monadic theory of topology defines more properties.
\begin{proposition}\label{defcountable}
  Let $\mc B$ be a sufficiently stable Boolean algebra such that every subset of $2^\N$ in $\mc B$ is determined.
  \begin{parts}
  \item The restricted monadic theory defines the class of countable sets.
  \item The restricted monadic theory defines the class of $F_\sigma$ sets.
  \item The restricted monadic theory defines the class of meager sets.
  \end{parts}
\end{proposition}
\begin{proof}
  (a) follows from the definability of perfect sets and the perfect set property. Similarly, (b) follows from Hurewicz' theorem \cite[Theorem 21.18]{kechrisdst}: A set $A$ is $G_\delta$ if and only if there is no Cantor set $C$ such that $A$ is countable dense in $C$. (c) follows as a set is meager if and only if it is included in an $F_\sigma$ set with empty interior.
\end{proof}
\begin{corollary}\label{decidpolish}
  The Borel monadic theory of topology of the class of zero dimensional Polish spaces is decidable.
\end{corollary}
\begin{proof}
  A space is zero dimensional Polish if and only if it is homeomorphic to a $G_\delta$ subset of the Cantor set. Since the class of $G_\delta$ sets is definable, the claim follows from decidability of the Borel monadic theory of topology of the Cantor set.
\end{proof}

\subsection{The Baire property}\label{secbaire}
It is natural to ask about decidability for classes between $\mathbf\Delta^1_1$ and $\mathbf\Delta^1_2$ in $\theo{ZFC}$. For example, decidability for the $\sigma$-algebra $\sigma(\mathbf\Sigma^1_1)$ generated by the analytic sets would follow from:
\begin{conjecture}\label{conjbaire}
  \Cref{mainres} continues to hold if determinacy is replaced by the Baire property.
\end{conjecture}
Since the Baire property for all projective sets is consistent with $\theo{ZFC}$ without requiring the large cardinal hypothesis of determinacy \cite[Theorem 7.16]{solovayinacc}, under \cref{conjbaire} decidability of the theory for projective sets is consistent with $\theo{ZFC}$. Similarly, since the consistency strength of determinacy is not necessary to obtain the Baire property for all sets \cite[Conclusion 7.17]{solovayinacc}, under \cref{conjbaire} decidability of the unrestricted monadic theory is independent of $\theo{ZF+DependentChoices}$.\par
The largest sufficiently stable Boolean algebra with the Baire property comprises the subsets $A\subseteq\R$ with the Baire property in the restricted sense. ($A\subseteq\R$ has the Baire property in the restricted sense if for all $C\subseteq\R$ the set $A\cap C$ has the Baire property in $C$. It suffices to verify this condition when $C$ is $\R$ or a Cantor set \cite[§11 VI]{kuratowskitopo}.) These sets form a $\sigma$-algebra containing all universally Baire sets.\par
We finally justify why \cref{conjbaire} is plausible. The only location in the proof where determinacy beyond the Baire property is used is \cref{secsepgame}, considering a generalized Wadge game. The lemma as stated implies the perfect set property and thus fails for $\sigma(\mathbf\Sigma^1_1)$ in some models of $\theo{ZFC}$. More generally, if \cref{conjbaire} is true, then every set without a perfect subset has the monadic type of a countable set. (For instance, one can show that no Bernstein set has the Baire property and so the sets without a perfect subset form an ideal under the assumptions of \cref{conjbaire}.) There is an analogous statement for $F_\sigma$ sets, whence \cref{defcountable} (a) and (b) fail for $\sigma(\mathbf\Sigma^1_1)$. However, we provide an alternative of (c) that only uses the Baire property in \cref{secdefinemeager}.\par
Extending a remark of \cite[p. 6]{msoplusu}, we show that the analogue of \cref{conjbaire} for $2^{\le\omega}$ fails. Indeed, there is a formula in $2^{\le\omega}$ equivalent to determinacy for $\mc B$ \cite[Section 5.3]{msoplusu}. This formula is true in the Borel monadic theory but false for Boolean combinations of analytic sets assuming the axiom of constructibility. Thus, we found two sufficiently stable Boolean algebras with the Baire property that yield different theories, contradicting the analogue of \cref{conjbaire}.

\section*{Acknowledgements}
This research was supported by the Hausdorff Center for Mathematics at the University of Bonn (DFG EXC 2047). I thank Philipp Hieronymi for suggesting the topic and support in writing this paper, Kim Kiehn and Michael Reitmeir for proofreading, Yuri Gurevich for discussing a question, and Taichi Yasuda for discussing the relation between my separation games and Wadge games.

\section{Notation and prerequisites}\label{sectoppre}
We start by introducing notation and some well-known results.\par
Let $X$ be a set. Let $\#X$ denote the cardinality of $X$. Let $\pows X$ denote the power set of $X$. The disjoint union or coproduct of sets $A,B$ is denoted by $A\amalg B$, which has lower precedence than cartesian products (that is, $A\amalg B\times C=A\amalg(B\times C)$).\par
For us, $0\in\N$. For $k\in\N^n$, write $\abs k=n$. For $k=(k_1,\dots,k_n)\in\N^n$ with $n>0$, write $k^-=(k_1,\dots,k_{n-1})$. We abbreviate $\set{0,\dots,m}^\N$ by $(m+1)^\N$. Thus, for $y\in m^\N$, we have $\limsup y\in\set{0,\dots,m-1}$.\par
A \textbf{labeling} of a set is a map to a finite set, whose elements are called labels. If $P$ is a labeling, then $\#P$ denotes the cardinality of its codomain. A labeling $Q\colon X\to J$ \textbf{refines} a labeling $P\colon X\to I$ \textbf{along} a map $f\colon J\to I$ if $P=f\circ Q$. In this case we say that $P$ \textbf{coarsens} $Q$. If $(X,P)$ is a labeled set, that is, $P$ is a labeling of $X$, and $Y\subseteq X$, then we abuse notation by writing $(Y,P)$ for the restriction $(Y,P|_Y)$. If $P$ is clear from the context, we often refer to properties of $(Y,P|_Y)$ or $P|_Y$ by saying that $Y$ has this property.\par
We usually consider labeled total orders. Embeddings and isomorphisms of these are required to preserve the labels. We thus obtain a category total orders labeled by a fixed set as objects and monotone label-preserving maps as morphisms. An isomorphism class in this category is called \textbf{labeled order type}.\par
Let $()$ denote the empty tuple. If $P$ is an $n$-tuple, we denote its components by $P_1,\dots,P_n$. If $P_1,\dots,P_n$ are elements, we denote the tuple of them by $P$. We often use notation for elements also for tuples, for example $\exists P$ for $\exists P_1\exists P_2\dots\exists P_n$.
Let $A,B$ be subsets of a topological space $X$. Let $\overline A,\mathring A,\partial A$ denote the closure, interior, and boundary of $A$. We call $A$ \textbf{perfect} if it is closed, nonempty, and has no isolated points. The set $A$ is \textbf{nowhere dense} if $\mathring{\overline A}$ is empty. It is \textbf{somewhere dense} else. We call $A$ \textbf{dense in} $B$ if $A\cap B$ is dense in the subspace $B$, and similarly for somewhere dense or (co-)meager sets.\par
A topological space is the \textbf{disjoint union} of subsets $(U_i)_{i\in I}$ if its underlying set is their disjoint union, but the $U_i$ need not be clopen. A family $(V_j)_{j\in J}$ of sets is a \textbf{refinement} of $(U_i)_{i\in I}$ if $\bigcup_jV_j=\bigcup_iU_i$ and for all $j$ there is an $i$ with $V_j\subseteq U_i$.\par
Let $\R$ denote the reals with the usual ordering. Let $\overline\C$ denote the Cantor set $\prod_{n\in\N}\set{0,1}$ with the lexicographic ordering and $\C=\overline\C\setminus\set{0^\infty,1^\infty}$ be obtained by removing endpoints. All total orders carry the order topology. Equivalently, $\overline\C$ carries the product topology of the discrete topologies and $\C$ carries the subspace topology of $\overline\C$.
\begin{fact}
  If $(Y,{\le})$ is a total order and $X\subseteq Y$ is compact in the subspace topology, then the order topology and the subspace topology on $X$ coincide.
\end{fact}
In particular, for subspaces homeomorphic to $\overline\C$, and thus also for subspaces homeomorphic to $\C$, the order topology and the subspace topology coincide. A subset of a space is \textbf{a Cantor set} if it is homeomorphic to $\C$. Note the deviation from the usual meaning of being homeomorphic to $\overline\C$.\par
\begin{definition}\label{cantorquot}
  Let $(C,\le)$ be a total order with $C\cong\C$. Consider the equivalence relation that identifies $a,b$ if $a=b$ or $(a,b)$ is a jump or $(b,a)$ is a jump. Let $\pi_C\colon C\to R$ denote the quotient map for this equivalence relation.
\end{definition}
In particular, $R\cong\R$ and $\pi_C$ is increasing. We sometimes also write $\pi_C\colon C\to\R$ for a composite with such an isomorphism.

\subsection{Descriptive set theory}
We collect basic facts of descriptive set theory. A reference is \cite{kechrisdst}.\par
A subset of a Polish space has the \textbf{Baire property} if it is the symmetric difference of an open set and a meager set. Borel sets, and more generally $\sigma$-combinations of analytic sets, have the Baire property.
\begin{fact}[Baire alternative]\labelfull{bairealt}{Baire alternative}
  Let $A$ have the Baire property. Then exactly one of the following conditions holds:
  \begin{tfae}
  \item $A$ is meager,
  \item there is a nonempty open $U$ such that $A$ is comeager in $U$.
  \end{tfae}
\end{fact}
Another useful immediate implication of Baire's theorem is the following fact.
\begin{fact}\label{closmeag}
  For a Cantor set $C$ in a Polish space, the following are equivalent:
  \begin{tfae}
  \item $C$ has empty interior,
  \item $C$ is nowhere dense,
  \item $C$ is meager.
  \end{tfae}
\end{fact}
\begin{proof}
  (i)$\iff$(ii) holds as $C$ is cofinite in its closure, (ii)$\implies$(iii) by definition and (iii)$\implies$(i) by Baire's theorem.
\end{proof}
Each nonempty totally disconnected compact metric space without isolated points is homeomorphic to $\overline\C$. This implies the following fact.
\begin{fact}\label{subcantoriff}
  \begin{parts}
  \item Let $A\subseteq\overline\C$. Then $A$ is homeomorphic to $\overline\C$ if and only if $A$ is compact, nonempty and has no isolated points.
  \item Let $A\subseteq\R$. Then $A$ is homeomorphic to $\overline\C$ if and only if $A$ is compact, nonempty, has no isolated points and empty interior.
  \end{parts}
\end{fact}
The next fact follows since open subsets of spaces without isolated points have no isolated points.
\begin{fact}\label{clopencantor}
  If $U\subseteq\overline\C$ is nonempty clopen, then $U$ is homeomorphic to $\overline\C$.
\end{fact}
Since any second countable topological space is the disjoint union of a closed subset without isolated points and a countable open subset, we obtain the following fact.
\begin{fact}\label{compmeagreal}
  \begin{parts}
  \item If $A\subseteq\C$ is compact, then $A$ is either countable or a disjoint union of a Cantor set and a countable set.
  \item If $A\subseteq\R$ is compact with empty interior, then $A$ is either countable or a disjoint union of a Cantor set and a countable set.
  \end{parts}
\end{fact}

\subsection{Order theory}
For some spaces, like $\R\setminus\Q$ or $\C$, the monadic theory of order is more expressive than the theory of topology. Thus, we need refined versions of the previous results concerning order isomorphisms and not only homeomorphisms.
\begin{fact}\label{cantorderisom}
  Let $A\subseteq\R$ or $A\subseteq\overline\C$ be a Cantor set. Then $A$ is order isomorphic to $\C$.
\end{fact}
\begin{proof}
  We may assume $A\subseteq\R$ as $\overline\C$ embeds into $\R$ simultaneously as order and as topological space. Then the fact is equivalent to the statement that each $A\subseteq\R$ homeomorphic to $\overline\C$ is order isomorphic to $\overline\C$. For a proof of this folklore result see \cite[Section 1.2, Theorem, p. 3]{orderCantor}.
\end{proof}
\begin{fact}\label{irratuniq}
  Let $A\subseteq\R$ be countable and dense. Then $\R\setminus A$ is order isomorphic to $\R\setminus\Q$.
\end{fact}
\begin{proof}
  By Cantor's theorem $A$ is order isomorphic to $\Q$. Extending such an isomorphism to completions, we obtain an order automorphism $f\colon\R\to\R$ with $f(A)=\Q$. Thus $f|_{\R\setminus A}$ is the desired isomorphism.
\end{proof}
Let $(X,\le)$ be a total order. We use interval notation $\inter{a,b},\ointer{a,b},\lointer{a,b},\rointer{a,b}$ for intervals in $X$. A \textbf{jump} in $X$ is a pair $(x,y)\in X\times X$ with $x<y$ such that there is no $z\in X$ with $x<z<y$. The jumps in $X$ are totally ordered by $(x,y)<(x',y')$ if and only if $x<x'$ if and only if $y<y'$ if and only if $y\le x'$. If $X$ is complete and dense and $C\subseteq X$ is closed, we call the connected components of $X\setminus C$ \textbf{complementary interval}s of $C$ in $X$. The complementary intervals are totally ordered by $I<J$ if and only if there are $i,j$ with $I\ni i<j\in J$ if and only if all $i\in I,j\in J$ satisfy $i<j$.
\begin{fact}\label{realsincantor}
  The set $\C$ has a cocountable subset order isomorphic to $\R$ and first-order definable in $(\C,\le)$.
\end{fact}
\begin{proof}
  Consider the set $X$ of $y\in\C$ such that there is no $x\in\C$ such that $(x,y)$ is a jump. Then $X$ is cocountable and definable. It is folklore that $X$ is order isomorphic to $\R$.
\end{proof}

\subsection{Ramsey theory}
\begin{definition}
  Let $(X,{\le})$ be a total order and $C$ be a set.
  \begin{parts}
  \item A \textbf{coloring} of $X$ by the set of colors $C$ is a map $f\colon\setb{(x,y)\in X\times X}{x<y}\to C$.
  \item A coloring $f$ is \textbf{additive} if $(C,+)$ is a semigroup such that $f(x,y)+f(y,z)=f(x,z)$ for all $x,y,z$.
  \item A subset $A\subseteq X$ is \textbf{homogeneous} if $f$ is constant on $(A\times A)\cap\dom f$.
  \end{parts}
\end{definition}
Let us state a particular case of Ramsey's theorem using this notation.
\begin{fact}\labelsthm{ramsey}{Ramsey}
  Every coloring of an infinite set by a finite set of colors has an infinite homogeneous subset.
\end{fact}
Additivity of a coloring may be rephrased as functoriality if the preorder $X$ and the monoid $C\cup\set0$ are regarded as categories as usual.
\begin{fact}\label{dloramsey}\cite[Theorem 1.3]{shelahord}
  Let $f$ be an additive coloring of an infinite dense total order $X$ by a finite semigroup $(C,+)$. Then there is a homogeneous subset of $X$ that is somewhere dense.
\end{fact}
Shelah's definition of additive coloring is superficially more general, but our version easily implies his version.
\begin{corollary}\label{cantorramsey}
  Let $f$ be an additive coloring of $\C$. Then $\C$ has a somewhere dense homogeneous subset.
\end{corollary}
\begin{proof}
  Apply \cref{dloramsey} to the subset of \cref{realsincantor}, which is dense.
\end{proof}

\subsection{Logic}\label{monadicviafo}
Let $\msc L$ be a \textbf{multi-sorted relational first-order language}, that is, a set $I$ of sorts, a set of relation symbols, and a map $a$ from the relation symbols to tuples of sorts. An atomic $\msc L$-formula in variables $x_k$ of sorts $i_k$ is a pair $R(x_{k_1},\dots,x_{k_n})$ of a relation symbol $r$ and a tuple of variables such that $(i_{k_1},\dots,i_{k_n})=a(R)$. General $\msc L$-formulas are built as usual from atomic ones using the Boolean connectives $\top,\land,\lnot$ and quantifiers $\exists$ for each sort, from which the symbols $\bot,\lor,\Rightarrow,\Leftarrow,\Leftrightarrow,\forall$ can be defined. Satisfaction of formulas can be defined recursively as usual. We will always assume that there is an equality relation symbol for each sort that is interpreted correctly.\par
\textbf{Definable} without further qualification means definable without parameters. Here we only study decidability and definability results, and will thus regard any two representations of a class of formulas as equivalent if there is a computable interpretation. In particular, we assume that all first-order languages are relational by replacing an $n$-ary function symbol by an $(n+1)$-ary relation symbol for its graph.\par
We now define monadic formulas and the monadic theory of a structure. Higher-order logic can be defined as multi-sorted first-order logic. For instance, to obtain the \textbf{monadic language} for a one-sorted first-order language $\msc L$ with sort $X$, we add a sort $\pows X$ for the power set and a binary relation symbol $\in$ on $X\times\pows X$. Thus, any first-order $\msc L$-structure induces a monadic $\msc L$-structure with the same universe. In this structure, it is easy to define the Boolean algebra structure on $\pows X$, for example the symbols $\subseteq,\emptyset,\cap,\cup,\setminus$. Moreover, being a singleton set is definable as being an atom in this Boolean algebra, and we introduce a predicate $\isAtom$ for this purpose.\par
However, it is also possible to reduce monadic logic to one-sorted first-order logic. In this case, the sort represents the power set, and we replace each relation symbol $R(x_1,\dots,x_m,Y_1,\dots,Y_n)$ with element variables $x_i$ and set variables $Y_j$ by the relation symbol $R'(X_1,\dots,X_m,Y_1,\dots,Y_n)$ interpreted as ``there are elements $x_1,\dots,x_m$ such that $X_i=\set{x_i}$ for all $i$ and $R(x_1,\dots,x_m,Y_1,\dots,Y_n)$ holds''. Clearly, this language defines the same relations on $\pows X$ as the original language and they interpret each other. Whenever we refer to usual first-order properties, like elementary embeddings, in the monadic setting, they are to be interpreted in this language.
\begin{remark}\label{reductionfouniv}
  There are alternative ways to express first-order relations $R$ as monadic relations. For example, we could replace $R(x_1,\dots,x_m,Y_1,\dots,Y_n)$ by a relation symbol $R'(X_1,\dots,X_m,Y_1,\dots,Y_n)$ with the interpretation ``for all $x_1\in X_1,\dots,x_m\in X_m$, the relation $R(x_1,\dots,x_m,Y_1,\dots,Y_n)$ holds.
\end{remark}
\textbf{Restricted monadic logic} can be handled in exactly the same way, replacing the power set $\pows X$ by a subset $\mc B$ in this discussion. In our discussions, $\mc B$ always contains all singletons, enabling the reduction to first-order logic, but it will not always be a Boolean subalgebra.\par
We use Gale-Stewart games, usually only called \textbf{games}, and their determinacy for Borel sets \cite{boreldet}. For standard notation, see also \cite[\nopp 20,21]{kechrisdst}.

\section{Constructing Cantor sets}\label{sectopo}
Consider a uniform labeling of the reals as defined in \cref{secsketch}. Uniformly labeled Cantor subsets with the same comeager label as the reals exist in all reasonable circumstances. The purely topological argument is provided here as \cref{findcantor}. Uniformly labeled Cantor subsets whose comeager label differs from the comeager label of the reals might not exist. Whenever they do not exist, this follows from Baire category arguments in \cref{secdecproofa}.
\begin{lemma}\label{goodnhd}
  Let $C\subseteq\R$ be compact, let $U$ be a neighborhood of $C$ and let $D_1,D_2\subseteq\R$ be dense. Then there are $a_1<b_1<a_2<\dots<a_n<b_n$ with $a_i\in D_1,b_i\in D_2$ such that $C\subseteq\bigcup_i\ointer{a_i,b_i}\subseteq U$.
\end{lemma}
\begin{proof}
  We may assume that $U$ is open. Then $U$ is a union of bounded open intervals. Since $C$ is compact, we may assume that the union is finite and disjoint. Let $a_1<b_1\le a_2<\dots\le a_n<b_n$ be such that $U=\bigcup_i\ointer{a_i,b_i}$. As $C$ is closed with $a_i,b_i\notin U\supseteq C$, there are $a_i',b_i'\in\R$ such that $a_i<a_i'<b_i'<b_i$ and $\rointer{a_i,a_i'},\lointer{b_i',b_i}$ are disjoint from $C$. Since $D_1$ and $D_2$ are dense, there are $a_i''\in D_1\cap\ointer{a_i,a_i'}$ and $b_i''\in D_2\cap\ointer{b_i',b_i}$. Clearly $C\subseteq\bigcup_i\ointer{a_i'',b_i''}\subseteq U$.
\end{proof}
Given a meager $F_\sigma$ set $A$ and Cantor sets $B_1,B_2,\dots$, we find a Cantor set $C$ disjoint from $A$ such that $C$ includes many $B_n$'s but $\bigcup_nB_n$ is meager in $C$. Moreover, we can control the jumps of $C$.
\begin{lemma}\label{findcantor}
  Let $A\subseteq\R$ be meager $F_\sigma$, let $B=\bigcup_{n\in\N}B_n$ be with each $B_n\subseteq\R$ closed nonempty, and let $D_1,\dots,D_\ell\subseteq\R$ be countable dense, such that $A,B_1,B_2,\dots,D_1,\dots,D_\ell$ are pairwise disjoint and each nonempty open subset of $\R$ includes some $B_n$. Let $E\subseteq\set{1,\dots,\ell}\times\set{1,\dots,\ell}$ be nonempty. Then there is a Cantor set $C\subseteq\R$ such that
  \begin{parts}
  \item $C$ is disjoint from $A$;
  \item $B$ is meager in $C$;
  \item each nonempty relatively open subset of $C$ includes some $B_n$;
  \item for each jump $(a,b)$ in $C$, there is $(i,j)\in E$ with $a\in D_i$ and $b\in D_j$;
  \item if $(i,j)\in E$, then the set of $a\in C$ such that there is a jump $(a,b)$ in $C$ with $a\in D_i$ and $b\in D_j$ is dense in $C$;
  \item every element of $C\cap\bigcup_iD_i$ is a component of a jump in $C$.
  \end{parts}
\end{lemma}
\begin{proof}
  Note that $B$ is meager. Indeed, otherwise there is $n_0$ such that $B_{n_0}$ has nonempty interior. Then no proper nonempty open subset of $B_{n_0}$ includes a $B_m$ by disjointness of the $B_n$. This contradicts that every nonempty open subset includes some $B_n$. Adding a countable dense subset disjoint from the meager $B\cup\bigcup_iD_i$, we may assume that $A$ is dense. Similarly, we may assume $\#B_n\ge2$ for all $n$. Namely, decompose $\N$ into pairs $\set{i,j}$ such that each nonempty open subset includes some $B_i\cup B_j$ and replace the $B_n$ by the $B_i\cup B_j$.\par
  Write $A=\bigcup_{n\in\N}A_n$ with $A_n$ compact nowhere dense and $D_i=\setb{d_{in}}{n\in\N}$.
  \begin{sublemma}\label{lemfindcantor}
    Let $e\colon\N\to E$. There are sequences $(S_n)_{n\in\N},(T_n)_{n\in\N}$ of subsets of $\R$ such that for all $n\in\N$
    \begin{enumerate}[label=(\roman*)]
    \item $S_n$ is closed, nowhere dense, nonempty, and disjoint from $A$;
    \item $\R\setminus T_n$ is a finite union of open intervals, and $T_n$ has no isolated points;
    \item for each complementary interval $\ointer{a,b}$ of $T_n$ we have that $a\in D_i$ and $b\in D_j$ for $e(k)=(i,j)$ with $k$ minimal such that $\ointer{a,b}$ is a complementary interval of $T_k$;
    \item $S_n\subseteq T_n$;
    \end{enumerate}
    and if $n>0$
    \begin{enumerate}[label=(\roman*),resume]
    \item $S_{n-1}\subseteq S_n$ and $T_n\subseteq T_{n-1}$;
    \item each complementary interval of $T_{n-1}$ is a complementary interval of $T_n$;
    \item if $U$ is a connected component of $T_{n-1}$, then $U\cap S_n$ includes some $B_i$;
    \item each connected component of $T_{n-1}$ is disconnected in $T_n$;
    \item $T_n$ is disjoint from $A_{n-1}\cup(\setb{d_{i,n-1}}{i\in\set{1,\dots,\ell}}\setminus\partial T_{n-1})$.
    \end{enumerate}
  \end{sublemma}
  \begin{proof}
    We construct $S_n,T_n$ by recursion on $n$. Start with an arbitrary $S_0\subseteq\R\setminus A$ with $S_0\cong\overline\C$ and $T_0=\R$.\par
    Given $S_n,T_n$ satisfying (i)-(ix), let $U_1,\dots,U_m$ be the connected components of $T_n$. By (ii) each $U_k$ includes an open interval. Hence, we can pick $B_{i_k}\subseteq U_k$. Set $S_{n+1}=S_n\cup\bigcup_{k=1}^mB_{i_k}$. Then (i) and (vii) hold.\par
    Pick $u_j\in\mathring{U_j}\setminus S_{n+1}$ for $j\in\set{1,\dots,m}$. By (iii), $\partial T_n\subseteq\bigcup_kD_k$. Since $A_n$ is disjoint from $\bigcup_kD_k$ and from $S_{n+1}$, we deduce that
    \[(A_n\cap T_n)\cup(\setb{d_{in}}i\cap\mathring{T_n})\cup\setb{u_j}j\subseteq\mathring{T_n}\setminus S_{n+1}.\]
    Write $(i_n,j_n)=e(n)$. \Cref{goodnhd} yields finitely many $a_\ell\in D_{i_n}\setminus\partial T_n,b_\ell\in D_{j_n}\setminus\partial T_n$ such that $a_1<b_1<a_2<\dots<b_L$ and
    \[(A_n\cap T_n)\cup(\setb{d_{in}}i\cap\mathring{T_n})\cup\setb{u_j}j\subseteq\bigcup_\ell\ointer{a_\ell,b_\ell}\subseteq T_n\setminus S_{n+1}.\]
    Set $T_{n+1}=T_n\setminus\bigcup_\ell\ointer{a_\ell,b_\ell}$.\par
    Part (ii) follows as $a_\ell,b_\ell\in\mathring{T_n}$ are pairwise distinct, (iii) and (v) are clear, and (iv) follows from
    \[S_{n+1}\subseteq S_n\cup\bigcup_kU_k=T_n\text{ and }S_{n+1}\cap\bigcup_\ell\ointer{a_\ell,b_\ell}\subseteq S_{n+1}\cap T_n\setminus S_{n+1}=\emptyset.\]
    Part (vi) follows from $\ointer{a_\ell,b_\ell}\subseteq T_n$ and (viii) from $\setb{u_j}j\subseteq\bigcup_l\ointer{a_\ell,b_\ell}$. Since
    \[A_n\cap T_{n+1}=A_n\cap T_n\setminus\bigcup_\ell\ointer{a_\ell,b_\ell}=\emptyset,\]
    part (ix) follows as
    \[\text{if }d_{jn}\notin\partial T_n\text{, then }d_{jn}\notin\overline{T_n}\supseteq T_{n+1}\text{ or }d_{jn}\in\setb{d_{in}}i\cap\mathring{T_n}\subseteq\bigcup_\ell\ointer{a_\ell,b_\ell}\subseteq\R\setminus T_{n+1}.\qedhere\]
  \end{proof}
  Let $e\colon\N\to E$ be such that every fiber is infinite. Let $S_n,T_n$ satisfy (i)-(ix). Let $C$ be $\bigcap_nT_n$ with minimum and maximum removed if they exist.\par
  Then (a) follows from (ix). For (c) let $a,b\in\R$ be such that $\ointer{a,b}\cap C\ne\emptyset$, say $x\in\ointer{a,b}\cap C$. Since $A$ is dense, there are $a',b'\in A$ with $a<a'<x<b'<b$. By (a) we obtain $n$ such that $a',b'\notin T_n$. Since $x\in C\subseteq T_n$, there is a connected component $U$ of $T_n$ with $x\in U\subseteq\ointer{a,b}$. By (vii) the set $\ointer{a,b}\cap C\supseteq U\cap S_{n+1}$ includes some $B_i$.\par
  Since $\#B_n\ge2$, by (c) the space $\overline C$ has no isolated points and (b) follows by the same argument as meagerness of $B$ in $\R$ above. Thus, as all $T_n$ are closed, $\overline C=\bigcap_nT_n$. Since $S_0\subseteq\overline C$, the set $C$ is nonempty. Since $A$ is dense, $\overline C$ has empty interior by (a). Using \cref{subcantoriff} in $\R\cup\set{-\infty,\infty}$, we conclude that $C$ is a Cantor set.\par
  By (ii) and (vi) an interval is a complementary interval of $C$ if and only if it is a complementary interval of some $T_n$. Thus (d) follows from (iii) and (f) from (ix). It remains to show (e). Since $C$ has empty interior, for each nonempty relatively open subset $V$ of $C$ for cofinitely many $n$ there is a connected component $U$ of $T_n$ with $U\cap C\subseteq V$. Pick such an $n$ with $e(n)=(i,j)$ and such an $U$. By (viii) and (iii), there is a complementary interval $\ointer{a,b}\subseteq U$ with $a\in D_i$ and $b\in D_j$. We conclude
  \[V\supseteq U\cap C\supseteq\inter{a,b}\cap C=\set{a,b}.\qedhere\]
\end{proof}
\begin{corollary}\label{findcantornoends}
   Let $A\subseteq\R$ be meager $F_\sigma$, and let $B=\bigcup_{n\in\N}B_n$ be with each $B_n\subseteq\R$ closed nonempty, such that $A,B_1,B_2,\dots$ are pairwise disjoint and each nonempty open subset of $\R$ includes some $B_n$. Then there is a Cantor set $C\subseteq\R$ such that
  \begin{parts}
  \item $C$ is disjoint from $A$.
  \item $B$ is meager in $C$.
  \item Each nonempty relatively open subset of $C$ includes some $B_n$.
  \end{parts}
\end{corollary}
\begin{proof}
  Let $\ell=1$ and $D_1$ be a countable dense subset disjoint from $A$ and $B$. Apply \cref{findcantor}.
\end{proof}

\subsection{Cantor subsets}
We present lemmas on the point set topology of the Cantor set.
\begin{lemma}
  Let $X$ be a Hausdorff space.
  \begin{parts}
  \item Any countable family of copies of $\overline\C$ in $X$ has a countable disjoint refinement consisting of copies of $\overline\C$.\label{refincantor}
  \item Suppose that $X$ is second countable. Let $(C_{ni})_{n\in\N,i\in I\amalg\tsingl}$ be a countable family of copies of $\overline\C$ in $X$ such that all $C_{ni}$ have empty interior and for all $i\in I$ each nonempty open subset includes some $C_{mi}$. Then the refinement of (a) can be chosen such that for each nonempty open $U\subseteq X$ there is a $C_{ni}\subseteq U$ that is also contained in the refinement.\label{refincantordense}
  \end{parts}
\end{lemma}
\begin{proof}
  For (a), consider a family $(C_n)_{n\in\N}$. Since $X$ is Hausdorff, each $C_n$ is closed. Thus each $C_n\setminus(C_0\cup\dots\cup C_{n-1})$ is open in $C_n$, hence a countable disjoint union of clopen subsets of $C_n$. By \cref{clopencantor} the family of all these clopen subsets is the desired refinement.\par
  For (b), since the $C_{ni}$ are nowhere dense, for all $i\in I$ and finite $A\subseteq\setb{C_{mj}}{m\in\N,j\in I}$ and nonempty open $U$ there is $n$ with $C_{ni}\subseteq U\setminus\bigcup_{C\in A}C$. Thus, using second countability, we may rearrange $(C_{ni})_{ni}$ as $(D_k)_{k\in\N}$ such that for all $i\in I$ and nonempty open $U$ there are $k,n\in\N$ with $C_{ni}=D_k\subseteq U\setminus(D_0\cup\dots\cup D_{k-1})$. If we carry out the construction of (a) with this enumeration, then the conclusion of (b) holds.
\end{proof}
\begin{lemma}\label{projcantorbij}
  Let $\pi_\C\colon\C\to\R$ be as in \cref{cantorquot}. Sending $C$ to its image $\pi_\C(C)$ defines a bijection from meager Cantor subsets of $\C$ to Cantor subsets of $\R$. The inverse maps $C$ to the Cantor-Bendixson derivative $\pi_\C^{-1}(C)'$ of the preimage.
\end{lemma}
\begin{proof}
  We instead show the equivalent statement that the extension $\pi\colon\overline\C\to\inter{0,1}$ of $\pi_\C$ yields a bijection between meager copies of $\overline\C$ in $\overline\C$ and copies of $\overline\C$ in $\inter{0,1}$. Note that $\pi$ is a continuous closed map with finite fibers. Preimages of arbitrary sets with empty interior have empty interior, and images of closed sets with empty interior have empty interior.\par
  We show that the map is well-defined. It suffices to show that $\pi(C)\cong\overline\C$ whenever $C\cong\overline\C$. By the above $\pi(C)$ is closed and nonempty and has empty interior. No $x$ is isolated in $\pi(C)$ as otherwise $\pi^{-1}(\set x)\cap C$ is nonempty, finite and open in $C$.\par
  We show that the inverse map is well-defined. It suffices to show that $\pi^{-1}(C)'\cong\overline\C$ whenever $C\cong\overline\C$. By the above $\pi^{-1}(C)$ and so $\pi^{-1}(C)'$ are closed and have empty interior. Since $C$ has no isolated points, each isolated point in $\pi^{-1}(C)$ is part of a two-element fiber (of a component of a jump) and the other element of the fiber is not isolated. Hence $\pi^{-1}(C)'$ has no isolated points.\par
  Moreover, this yields $\pi(\pi^{-1}(C)')=\pi(\pi^{-1}(C))=C$.\par
  Clearly $\pi^{-1}(\pi(C))'\supseteq C'=C$, so suppose that $x\in\pi^{-1}(\pi(C))'$. Then $\pi(x)=\pi(y)$ for some $y\in C$. If $x=y$, we are done, and else by symmetry assume $x<y$. Then, as $x$ is not isolated, each open interval $\ointer{z,x}$ meets $\pi^{-1}(\pi(C))$ in uncountably many points and, in particular, in some point that is no component of a jump. Thus $\ointer{z,x}$ meets $C$ and so $x\in\overline C=C$.
\end{proof}

\subsection{Iso-uniform orders}
Here we provide a first construction of uniform sums, which satisfies a stricter uniformity property involving isomorphism classes instead of finite partial types. The description for partial types is in \cref{secdecproofa}.
\begin{definition}
  Let $X\subseteq\R$ be locally closed and let $P$ be a labeling of $X$. We call $P$ \textbf{iso-uniform} if all nonempty convex $I,J\subseteq X$ without endpoints satisfy $(I,P)\cong(J,P)$ as labeled orders.
\end{definition}
If $P$ is clear from the context, we say that $X$ is \textbf{iso-uniformly labeled}. In particular, if $Y\subseteq X$ is locally closed, then we say that $Y$ is iso-uniformly labeled if $P|_Y$ is iso-uniform. Similar conventions apply to notions as $k$-uniformity introduced later.
\begin{definition}
  Let $T$ be a finite set of labeled order types of Cantor sets and points. Let $P$ be a labeling of $X\cong\R$ and $c_0$ be a label. We call $(X,P)$ an \textbf{iso-uniform sum} of $T$ with comeager label $c_0$ if
  \begin{parts}
  \item for every nonempty open $U\subseteq X$ and $t\in T$, there is a $C\subseteq U$ such that $(C,P)$ realizes $t$;\label{defisounifsup}
  \item there are countably many closed nowhere dense $C_i\subseteq X$ such that for each $i$ there is $t\in T$ such that $(C_i,P)$ can be embedded into a realization of $t$ and $\R\setminus\bigcup_iC_i$ has label $c_0$.\label{defisounifsub}
  \end{parts}
\end{definition}
\begin{proposition}
  Let $T$ be a finite set of labeled order types of Cantor sets and points and $c_0$ be a label.
  \begin{parts}
  \item Iso-uniform sums of $T$ are iso-uniformly labeled.
  \item There is an iso-uniform sum of $T$ with comeager label $c_0$.
  \item All iso-uniform sums of $T$ with comeager label $c_0$ are isomorphic.
  \end{parts}
\end{proposition}
\begin{proof}
  If $T$ is empty, all statements are clear, so assume $T$ nonempty. We may also assume that no $t\in T$ has $c_0$ as single realized label as otherwise the iso-uniform sums of $T$ are precisely the iso-uniform sums of $T\setminus\set t$.\par
  For (c), let $P,P'$ be iso-uniform sums of $T$ with comeager label $c_0$, with $(C_i)_{i\in\N},(C_i')_{i\in N}$ as in \cref{defisounifsub}. We first define by recursion partial isomorphisms $f_i\colon A_i\cong A_i'$ with $A_i\supseteq C_0\cup\dots\cup C_{i-1}$ and $A_i'\supseteq C_0'\cup\dots\cup C_{i-1}'$ closed nowhere dense.\par
  Start with $A_0=A_0'=\emptyset$. Given $f_i$, we obtain an order isomorphism $g_i$ between complementary intervals of $A_i$ and $A_i'$ since these are given by jumps in $A_i$ and $A_i'$. For each complementary interval $J$ of $A_i$, by definition $(C_i\cap J,P)$ embeds into a realization of some $t\in T$. The interval $g_i(J)$ includes a realization $D$ of $t$ by \cref{defisounifsup}. Extend $f_i$ by mapping $C_i\cap J$ into $D$ for each $J$. Since $J$ is a complementary interval, this is a partial isomorphism. Repeating this with the inverse of this extension yields $f_{i+1}$.\par
  Set $A=\bigcup_iA_i$ and $A'=\bigcup_iA_i'$. The $f_i$ combine to $f\colon A\cong A'$. By assumption on $T$, every nonempty open subset includes some $C_i$ and some $C_i'$. Thus, $A$ and $A'$ are dense in $\R$. By uniqueness of Dedekind completions, $f$ can be extended to an order automorphism of $\R$ (explicitly, $x\mapsto\sup_{A\ni a\le x}f(a)$). This extension is still an isomorphism of labelings as both $\R\setminus A$ and $\R\setminus A'$ have label $c_0$.\par
  If $(X,P)$ is an iso-uniform sum of $T$ and $Y\subseteq X$ is nonempty convex, then $(Y,P)$ is an iso-uniform of $T$ with the same comeager label. Thus, (a) follows from (c). It remains to prove (b).\par
  For $t\in T$, let $\overline t$ be the labeled order type that equals $t$ if $t$ describes a point and is obtained by adding endpoints of label $c_0$ to $t$ else. Let $C=\sum_{t\in T}\overline t$ be the labeled order type of their lexicographic sum. Let $Q$ be the order type of the jumps of $C$. For a closed embedding $C\to\ointer{0,1}$ of orders, $Q$ equals the order type of the complementary intervals of the image.\par
  We define sets $\mc C_n$ of pairwise disjoint realizations of $C$ and a map $f$ from finite sequences in $Q$ to nonempty convex subsets of $\ointer{0,1}$ by simultaneous recursion. Start with $\mc C_0=\emptyset$ and $f()=\ointer{0,1}$. Given $q=(q_1,\dots,q_n)$, since $T$ is nonempty, we may take a realization $D_q$ of $C$ in $f(q)$ all whose complementary intervals have length at most $\frac12\diam f(q)$. For $r\in Q$, let $f(q,r)$ be the $r$-th complementary interval of $D_q$. Set $\mc C_{n+1}=\setb{D_q}{\abs q\le n}$.\par
  We now define a labeling $P$ of $\ointer{0,1}$. If $x\in\bigcup_n\mc C_n$, label $x$ as in $C$, and else label $x$ with $c_0$. For showing that $(\ointer{0,1},P)$ is an iso-uniform sum, decompose all summands of elements of $\bigcup_n\mc C_n$ into countably many points and copies of $\overline\C$, and let the $C_i$ enumerate all these sets.
\end{proof}

\subsection{Definability in restricted monadic theories}\label{secdefinemeager}
Adding endpoints to $\R,\R\setminus\Q$ or $\C$ does not change the computational power of their restricted monadic theories. This is straightforward. One implication is also a particular case of \cref{binsumdecid}.
\begin{proposition}\label{equivcomp}
  The restricted theories of $\R,\R\setminus\Q,\C$ are computable from each other.
\end{proposition}
\begin{proof}
  To interpret $\R\setminus\Q$ in $\R$, use \cref{irratuniq}. To interpret $\overline\C$ in $\R\setminus\Q$, use \cref{subcantoriff} and that compactness of a subset of $\R\setminus\Q$ is definable as it is equivalent to being a complete total order with endpoints, and a total order is complete if and only if every convex subset is an interval. To interpret $\R$ in $\C$, use \cref{realsincantor}.
\end{proof}\par
We next provide an alternative proof of definability of the class of meager sets that works in every sufficiently stable Boolean algebra whose members have the Baire property. This provides an easier example of our methods and motivates adding meagerness to the language in \cref{secexpandlang}.\par
\begin{lemma}\label{lemdefmeagersmall}
  Let $X$ be a Polish space without isolated points and let $A\subseteq X$ be meager. There is a countable dense $B\subseteq X$ such that there is no Cantor set $C\subseteq A\cup B$ with $B$ dense in $C$.
\end{lemma}
\begin{proof}
  Let $A'\supseteq A$ and $A_i$ be closed nowhere dense with $A'=\bigcup_iA_i$. Then there is a countable dense $B\subseteq X\setminus A'$. Suppose that $C\subseteq A\cup B$ is a Cantor set. Then some $A_i$ is nonmeager in $C$ and thus has nonempty interior in $C$. Thus, $B$ is not dense in $C$ as $A_i\cap B\subseteq A'\cap B=\emptyset$.
\end{proof}
\begin{lemma}\label{lemdefmeagerlarge}
  Let $A\subseteq\R$ be comeager. For all dense $B\subseteq\R$ there is a Cantor set $C\subseteq A\cup B$ with $B$ dense in $C$.
\end{lemma}
\begin{proof}
  We may assume that $A$ is $G_\delta$ and $B=\setb{b_n}n$ is countable. Then \cref{findcantornoends} with $\R\setminus(A\cup B)$ as $A$ and $\set{b_n}$ as $B_n$ yields $C$.
\end{proof}
\begin{proposition}
  Let $A\subseteq\R$ have the Baire property and $\mc B\subseteq\pows\R$ contain all countable sets. Then the following are equivalent:
  \begin{tfae}
  \item $A$ is meager.
  \item For all nonempty open $U\subseteq\R$ there is a dense subset $B\in\mc B$ of $U$ such that there is no Cantor set $C\subseteq A\cup B$ with $B$ dense in $C$.
  \end{tfae}
\end{proposition}
\begin{proof}
  If $A$ is meager, use \cref{lemdefmeagersmall} with $U$ for $X$ and $A\cap U$ for $A$. Else, by the \cref{bairealt} choose an open interval $I$ with $A$ comeager in $I$ and apply \cref{lemdefmeagerlarge} to $I\cong\R$ with $A\cap I$ for $A$.
\end{proof}

\section{Shelah's decidability technique}\label{secrepshe}
This section repeats the results of \cite{shelahord} concerning decidability for dense total orders. They are rephrased in our language, straightforwardly extended to restricted quantifiers, and specialized to complete orders for the sake of simplicity. For the first part, other references are \cite{monadicsurvey,modchaini}. We start with notation.\par
Let $A,B$ be sets together with fixed injections to $\N$. We call a map $A\to B$ \textbf{computable} if there is a computable map $\N\to\N$ making the diagram
\[\begin{tikzcd}
A\ar{r}\ar{d}&B\ar{d}\\
\N\ar[dashed]{r}{\exists}&\N
\end{tikzcd}\]
commute. We will use this notation for example if $A\subseteq\N^{<\omega}$ or if $A$ is a (not necessarily computable) set of formulas in a finite language, without explicitly mentioning the injection to $\N$.\par
A family of functions $(f_i)_i$ is \textbf{computable uniformly in} $i$ if the map $(i,x)\mapsto f_i(x)$ is computable. Analogously, a family of sets is \textbf{computable uniformly in} $i$ if their indicator functions are so, and a sequence of theories is \textbf{decidable uniformly in} $i$ if the set of valid formulas is computable uniformly in $i$.
\begin{definition}\label{welldefcomp}
  Let $X$ be a class and $f\colon X\to A,g\colon X\to B$ be maps, where $A,B$ are sets together with fixed injections to $\N$. (Below, $X$ is a class of structures, or families of structures, $A,B$ are sets of formulas, and $f,g$ map a structure to its partial theory.)\par
  Let $f(x)\mapsto g(x)$ denote the relation with graph $\setb{(f(x),g(x))}{x\in X}$. We say that $f(x)\mapsto g(x)$ is \textbf{well-defined} if this relation is a partial function. Similarly, the family $f_i(x)\mapsto g_i(x)$ is \textbf{well-defined} if for all $i$ the relation $f_i(x)\mapsto g_i(x)$ is a partial function. We usually specify $X$ by saying that the map is \textbf{well-defined in} $x$ where the variable $x$ is known to range over $X$.
\end{definition}
When proving that such functions are well-defined and computable, we usually describe a computable extension $A\to B$ of $f(x)\mapsto g(x)$ informally by describing how to compute $g(x)$ only depending on $f(x)$. In our settings, the image of $f$ will not generally be a computable subset of $A$.

\subsection{Finite partial types}
We first describe Shelah's constructions in a first-order setting before specializing to the monadic setting.\par
We call a first-order theory $T$ in a finite language \textbf{essentially relational} if there is a computable map $f$ from atomic formulas to quantifier-free formulas in the same free variables such that $T\models\phi\Leftrightarrow f(\phi)$ for all $\phi$ and for each fixed finite set $V$ of variables the image of formulas with free variables in $V$ under $f$ is finite. For instance, theories in relational languages are essentially relational via $f=\id$. The theory of Boolean algebras in the language $\set{{\le},\vee,\wedge,\lnot,\bot,\top}$ is essentially relational via $f$ replacing all terms by their disjunctive normal forms.
\begin{definition}
  Let $T$ be a first-order theory and let $f$ witness that $T$ is essentially relational. Let $A\models T$, $a\in A^m$ and $k\in\N^n$.
  \begin{parts}
  \item We define $\fTh^-_k(m)$ by recursion on $n$:
    \begin{itemize}
    \item For $n=0$, $\fTh^-_{()}(m)$ is the image under $f$ of the atomic formulas with $m$ free variables.
    \item If $n>0$, then $\fTh^-_k(m)$ is the power set of $\fTh^-_{k^-}(m+k_n)$.
    \end{itemize}
  \item We define $\Th^1_k(A,a)\subseteq\fTh^-_k(m)$, the $k$-\textbf{theory} of $(A,a)$, by recursion on $n$:
    \begin{itemize}
    \item For $n=0$, set $\Th^1_{()}(A,a)=\setb{\phi\in\fTh^-_{()}(m)}{(A,a)\models\phi}$.
    \item For $n>0$, set $\Th^1_k(A,a)=\set{\Th^1_{k^-}(A,ab)\mid b\in A^{k_n}}$.
    \end{itemize}
    We abbreviate $\Th^1_k(A)=\Th^1_k(A,())$.
  \item Let $\fTh^1_k(m)=\fTh^-_{(k,0)}(m)$ be the set of \textbf{formally possible $k$-theories}.
  \end{parts}
\end{definition}
This definition deviates from the preliminary definition of $\Th$ in \cref{secsketch} for technical convenience. We first observe how the two definitions are equivalent. To this end, define by induction on $k\in\N^n$ formulas of \textbf{quantifier type} $k$. A formula has quantifier type $()$ if it is atomic. For $n>0$, a formula has quantifier type $k$ if it is of the form $\exists x_1\dots\exists x_{k_n}\phi$ with $\phi$ a Boolean combination of formulas of quantifier type $k^-$. By induction, $\Th^1_k(A,a)=\Th^1_k(B,b)$ if and only if $(A,a)$ and $(B,b)$ satisfy the same formulas of quantifier type $k$.\par
Thus, if we call a partial type in $T$ $k$-complete whenever it only consists of formulas of quantifier type $k$ and is maximal among such types, then elements of $\Th^1_k(A)$ correspond to $k^-$-complete types in $k_n$ free variables. Since the Boolean algebras generated by formulas with fixed quantifier type are finite, all filters are principal. Types, that is, filters of formulas, are thus determined by single formulas. Hence, we will use the terms ``type'' and ``theory'' interchangeably for these $\Th^1_k$.\par
Furthermore, every $\Th^1_k(A)\in\fTh^1_k(0)$ is hereditarily finite and $(k,m)\mapsto\fTh^1_k(m)$ is computable. Computability of $k\mapsto\Th^1_k(A)$ is equivalent to decidability of the first-order theory of $A$.
\begin{definition}
  Let $X$ be a class, let $x\in X$, let $F$ be a map from $X$ to a class of $\msc L_F$-structures and $G$ be a map from $X$ to a class of $\msc L_G$-structures. We say that the maps are $\Th^1_{O(k)}(F(x))\mapsto\Th^1_k(G(x))$ are \textbf{well-defined in }$x$\textbf{ and computable} if there is a computable $\ell\colon\N^{<\omega}\to\N^{<\omega}$ such that $\abs{\ell(k)}\le\abs k$ for all $k$ and the maps $\Th^1_{\ell(k)}(F(x))\mapsto\Th^1_k(G(x))$ are well-defined in $x$ and computable uniformly in $k$.
\end{definition}
\begin{example}\label{exampledlo}
  The well-known completeness, decidability and quantifier elimination of the theory of dense total orders without endpoints can be rephrased in this setting, simultaneously computing types. Let $A$ be a dense total order without endpoints and $a\in A^n$.\par
  We first describe the quantifier-free types. Namely, there is a computable bijection $\Phi$ from the set of possible $\Th^1_{()}(A,a)$ onto the set of surjections $\set{1,\dots,n}\to\set{1,\dots,m}$ for some $m$. Here, $\Phi(\Th^1_{()}(A,a))$ is the unique $f$ such that $a_i\mapsto f(i)$ is a well-defined order isomorphism $\set{a_1,\dots,a_n}\cong\set{1,\dots,m}$.\par
  Quantifier elimination states that $\Th^1_{()}(A,a)\mapsto\Th^1_k(A,a)$ is well-defined in $A,a$. By recursion, it suffices to consider $\Th^1_{()}(A,a)\mapsto\Th^1_{(k_1)}(A,a)$. We describe the image $\Phi(\Th^1_{(k_1)}(A,a))=\setb{\Phi(\Th^1_{()}(A,ab))}{b\in A^{k_1}}$ in terms of $\Phi(\Th^1_{()}(A,a))$. Namely, $f\colon\set{1,\dots,n+k_1}\to\set{1,\dots,m}$ equals $\Phi(\Th^1_{()}(A,ab))$ for some $b$ if and only if $f|_{\set{1,\dots,n}}=\Phi(\Th^1_{()}(A,a))$. Necessity is clear and sufficiency follows since $A$ is dense without endpoints.\par
  This also describes types: Each type is determined by $\Th^1_{()}(A,a)$ and thus corresponds to a surjection $\set{1,\dots,n}\to\set{1,\dots,m}$. In particular, since there is a unique type of the empty tuple, the theory is complete. Finally, decidability follows from computability of $\Th^1_{()}(A,a)\mapsto\Th^1_k(A,a)$.
\end{example}
This formalism also applies to restricted monadic theories in the language of order (or any finite relational language) by our previous reduction to first-order theories in \cref{monadicviafo}. In other words, we work in expansions of the language $({\le},{\subseteq},\cup,\cap,\setminus,X,\isAtom)$ of Boolean algebras with a predicate for atoms and an order $\le$ on the atoms. Instead of $\Th^1_k$ as for first-order theories, we denote partial monadic theories by $\Th_k$.\par
If $Y$ is a substructure of a restricted monadic structure $X$ and $P=(P_1,\dots,P_n)$ is a tuple of subsets of $X$, then we abuse notation by writing
\[\Th_k(Y,P)=\Th_k(Y,(Y\cap P_1,\dots,Y\cap P_n)).\]
In this case, if $(X,P)$ and $k$ are clear from the context, we often use ``the type of $Y$'' to denote $\Th_k(Y,P)$.\par
Although equivalent, it is often more convenient to work with labelings instead of subsets.
\begin{definition}
  Let $(X,{\le})$ be a total order, regarded as restricted monadic structure with quantification over $\mc B$. A $\mc B$-\textbf{labeling} of $X$ is a map $P\colon X\to I$ to a finite set $I$ such that all fibers lie in $\mc B$. We also write $P_i$ for the fiber $P^{-1}(\set i)$ and $\Th_k(X,P)$ for $\Th_k(X,(P_i)_{i\in I})$.
\end{definition}
Alternatively, a labeling could be encoded by adding a new sort and the map itself, leading to the same expressive power but not eliminating some quantifiers. Given $\Th_{()}(X,(P_i)_{i\in I})$ for a tuple $(P_i)_i$ of subsets, it is computable whether the $P_i$ assemble into a labeling $X\to I$. Namely, they do if and only if they are pairwise disjoint with $\bigcup_iP_i=X$. For now, we assume that the only expansions of the language are by tuples of subsets or labelings.

\subsection{Monadic sums}
Next, we show how to compute the monadic theory of a lexicographic sum given the theories of its summands and the index set in an appropriate sense. Let $(I,\le)$ be a total order and for $i\in I$ let $(X_i,\le)$ be total orders. We regard them as restricted monadic structures with quantification over $\mc B_i\le\pows{X_i}$ and $I$ as monadic structure with unrestricted quantification.\par
Let $\sum_{i\in I}X_i$ be the total order with underlying set the disjoint union $\coprod_iX_i$ and the lexicographic ordering. We interpret monadic constant symbols $P$ by the union of their interpretations $P_i$ in the summands. In other words, we define the sum of orders with a distinguished subset by
\[\sum_i(X_i,P_i)=\pare{\sum_iX_i,\bigcup_iP_i}.\]
Finally, $\sum_iX_i$ is regarded as a monadic structure with quantification restricted to elements of
\[\mc B_{\sum_iX_i}=\set{\bigcup_iA_i\mid A_i\in\mc B_i}.\]
If $I$ is countable, all $X_i\subseteq\R$ are $G_\delta$ and $\mc C$ is a sufficiently stable family of Boolean algebras, then, since $\set{\bigcup_iA_i\mid A_i\in\mc C_{X_i}}=\mc C_{\sum_iX_i}$, our definitions coincide.
\begin{theorem}\cite[Theorem 2.4]{shelahord}\label{monadicsum}
  Let $I$ be a total order and $(X_i,P_i)$ for $i\in I$ be total orders with tuples of $m$ subsets. Consider the labeling
  \[T^{kP}\colon I\to\fTh_k(m),i\mapsto\Th_k(X_i,P_i).\]
  Then the maps $\Th_{O(k)}(I,T^{kP})\mapsto\Th_k(\sum_i(X_i,P_i))$ are well-defined in $I,X_i,P_i$ and computable.
\end{theorem}
In particular, as the theory of the summands determines the theory of the sum uniquely, the lexicographic sum induces well-defined operations $\sum_{i\in I}\colon(\fTh_k(m))^I\to\fTh_k(m)$ with
\[\sum_{i\in I}\Th_k(X_i,P_i)=\Th_k\pare{\sum_{i\in I}(X_i,P_i)}.\]
\begin{proof}
  Write $(X,P)=\sum_i(X_i,P_i)$. Following our conventions in \cref{welldefcomp}, we only need to effectively describe $\Th_k(X,P)$ in terms of $\Th_\ell(I,T^{kP})$ for some $\ell$ with $\abs{\ell}=k$ by recursion on $\abs k$. Recall that $(T^{kP})_t$ denotes the fiber $(T^{kP})^{-1}(\set t)$. Thus, $i\in(T^{kP})_t$ if and only if $\Th_k(X_i,P_i)=t$.\par
  For $\abs k=0$, let $s$ and $s'$ be terms in the monadic language of order. We distinguish cases for atomic formulas.\par
  First, $X\models s\subseteq s'$ if and only if all $t\in\fTh_k(m)$ with $(s\nsubseteq s')\in t$ satisfy $I\models(T^{kP})_t=\emptyset$. Indeed, after unraveling definitions this states that $X\models s\subseteq s'$ if and only if $\setb i{X_i\models s\nsubseteq s'}=\emptyset$. The argument for $s=s'$ is analogous.\par
  Next, $X\models\isAtom(s)$ if and only if there is a unique $t\in\fTh_k(m)$ such that
  \begin{itemize}
  \item $\isAtom(s)\in t$;
  \item $I\models\isAtom((T^{kP})_t)$;
  \item all $t'\in\fTh_k(m)$ with $t'\ne t$ satisfy $(s=\emptyset)\in t'$ or $I\models(T^{kP})_{t'}=\emptyset$.
  \end{itemize}\par
  We finish the base case by describing when $X\models s\le s'$. We require that $X\models\isAtom(s)$ and $X\models\isAtom(s')$, say witnessed by $t,t'\in\fTh_k(m)$ with $\isAtom(s)\in t$ and $\isAtom(s')\in t'$ and $I\models\isAtom((T^{kP})_t)$ and $I\models\isAtom((T^{kP})_{t'})$. Then, we demand that
  \begin{itemize}
  \item $I\models(T^{kP})_t<(T^{kP})_{t'}$; or
  \item both $t=t'$ and $(s\le s')\in t$.
  \end{itemize}\par
  Finally, suppose that $\abs k>0$. Given $\Th_{O(k)}(I,T^{kP})$, we need to compute the set of possible $\Th_{k^-}(X,PP')$ as $P'$ varies over tuples of subsets of $X$. Since $\Th_{k^-}(X,PP')=\sum_i\Th_{k^-}(X_i,PP')$, by recursion it suffices to compute the set of possible $\Th_{O(k^-)}(I,T^{k^{-},PP'})$. We thus describe for a labeling $f$ of $I$ in terms of $\Th_{O(k^-)}(I,T^{kP},f)$ whether $f\in\setb{T^{k^-,PP'}}{P'}$. Such a $P'$ corresponds to tuples $P_i'$ of subsets of $X_i$, whose theories comprise $T^{kP}(i)=\Th_k(X_i,P_i)$. Thus,
  \[f\in\setb{T^{k^-,PP'}}{P'}\text{ if and only if }f(i)\in T^{kP}(i)\text{ for all }i\text{ if and only if }f_t\subseteq\bigcup_{t'\ni t}(T^{kP})_{t'}\text{ for all }t.\]
  This can be computed even from $\Th_{()}(I,T^{kP},f)$.
\end{proof}
\begin{corollary}\label{binsumdecid}
  Let $(X,P),(Y,Q)$ be labeled total orders. The maps
  \[(\Th_k(X,P),\Th_k(Y,Q))\mapsto\Th_k(X+Y,P+Q)\]
  are well-defined in $X,Y,P,Q$ and computable uniformly in $k$.
\end{corollary}
\begin{proof}
  Compose $(\Th_k(X,P),\Th_k(Y,Q))\mapsto\Th_{O(k)}(\set{1,2},T^{kP})$ with the map of \cref{monadicsum}. Since $T^{kP}(1)=\Th_k(X,P)$ and $T^{kP}(2)=\Th_k(Y,P)$ are given and $\set{1,2}$ is finite, the type $\Th_{O(k)}(\set{1,2},T^{kP})$ can be computed.
\end{proof}
To illustrate these methods, we reprove decidability of S1S. We only work with unrestricted monadic theories.
\begin{lemma}\label{finorddecid}
  \begin{parts}
  \item The monadic theory of $\set{1,\dots,n}$ is decidable uniformly in $n$.
  \item The common monadic theory of the class of finite total orders is decidable.
  \end{parts}
\end{lemma}
\begin{proof}
  The theories of $\emptyset,\set1$ are decidable by finiteness. The theory of $\set{1,\dots,n}$ can be computed uniformly by iterating the algorithm of \cref{binsumdecid}.\par
  For (b), let $k$ be a tuple of subsets. By the pigeonhole principle, $\Th_k(\set{1,\dots,p})=\Th_k(\set{1,\dots,p+q})$ for some $p,q\in\N$ with $q>0$. It suffices to show that each finite total order $I$ has the same $\Th_k$ as some $\set{1,\dots,n}$ with $n<p+q$. Write $\#I=p+mq+r$ with $r<q$. Then $n=p+r$ works since we can prove by induction on $m$ that
  \[\Th_k(I)=(\Th_k(\set{1,\dots,p})+m\Th_k(\set{1,\dots,q}))+\Th_k(\set{1,\dots,r})=\Th_k(\set{1,\dots,p})+\Th_k(\set{1,\dots,r}).\qedhere\]
\end{proof}
\begin{theorem}[S1S]\cite[Theorem 3.4]{shelahord}
  The monadic theory of $(\N,\le)$ is decidable.
\end{theorem}
\begin{proof}
  We compute $\Th_k(\N)$ by recursion on $\abs k$, and $\abs k=0$ is trivial. Suppose that $\abs k=n>0$.\par
  It suffices to show that $\Th_k(\N)=\setb{\Th_{k-}(\N,P)}{P\in(\pows\N)^{k_n}}$ is the set of $s+\sum_{i\in\N}t$ for $s,t$ ranging over all $k^-$-types of nonempty finite total orders with a tuple of $k_n$ subsets. Indeed, the set of all possible $s,t$ can be computed by \cref{finorddecid}, and the map $(s,t)\mapsto s+\sum_it$ can be computed by \cref{monadicsum}. Indeed, the computation of $t\mapsto\sum_it$ requires $\Th_{O(k^-)}(\N,T^{k^-,P})$ for the constant map $P\colon i\mapsto t$. It is computable from $\Th_{O(k^-)}(\N)$ and thus by recursion.\par
  Clearly, all $s+\sum_it$ are possible types of tuples of subsets of $\N$. Conversely, let $P\in(\pows\N)^{k_n}$ be arbitrary. Color pairs of $i<j$ by $t_{ij}=\Th_{k^-}(\set{i,\dots,j-1},P)$. By \cref{ramsey}, there are $t$ and an infinite set of $n_0<n_1<\dots$ such that $t_{n_i,n_{i+1}}=t$ for all $i$. Then $\Th_{k^-}(\N,P)=t_{0,n_0}+\sum_it$.
\end{proof}
\begin{corollary}
  The monadic theories of the order dual $\dual{\N}$ and of $\Z$ are decidable.
\end{corollary}
\begin{corollary}
  Let $X$ be a total order. The maps $\Th_k(X)\mapsto\Th_k(\sum_{i\in\N}X)$ and $\Th_k(X)\mapsto\Th_k(\sum_{i\in\Z}X)$ are well-defined in $X$ and computable.
\end{corollary}

\subsection{Special quantifiers}
Using monadic quantifiers, quantifiers over elements or labelings can be expressed. We now define analogues of $\Th_k$ for such quantifier types.
\subsubsection{First-order quantifiers}\label{secfoquant}
\begin{definition}
  Let $\Phi$ be an arbitrary map whose domain are structures $X$ with $\mc B\subseteq\pows X$ and $P\in\mc B^k$. Let $m\in\N$. Set $\asub_m(\Phi)(X,P)=\setb{\Phi(X,PP')}{P'\in\mc B^m}$.
\end{definition}
Informally, $\asub$ describes the standard restricted monadic quantifier. More precisely, $\Th_k=\asub_{k_n}(\Th_{k^-})$, hence inductively $\Th_k=\asub_{k_n}(\asub_{k_{n-1}}(\dots(\asub_{k_1}(\Th_{()})\dots)))$.\par
Next, we want to combine first-order theories with monadic theories. If $\mc B$ contains all singletons $\set x$, define $\elem(\Phi)(X,P)=\set{\Phi(X,P\set x)\mid x\in X}$. For $m\in\N,k\in\N^n$ define the mixed theory
\[\Th_{m;k}=\asub_{k_n}(\asub_{k_{n-1}}(\dots(\asub_{k_1}(\underbrace{\elem(\dots(\elem}_{m\text{ times}}(\Th_{()})\dots)))))).\]
Our previous $\Th_k$ is $\Th_{0;k}$. We obtain an analogue of \cref{monadicsum}.
\begin{theorem}\label{monadicsumfo}
  Let $I$ be a total order and $(X_i,P_i)$ for $i\in I$ be total orders with tuples of subsets. Let $T^{m;k,P}\colon i\mapsto\Th_{m;k}(X_i,P_i)$. The maps $\Th_{m;O(k)}(I,T^{m;k,P})\mapsto\Th_{m;k}(\sum_i(X_i,P_i))$ are well-defined in $I,X_i,P_i$ and computable.
\end{theorem}
\begin{proof}
  We again work by recursion on $(m,k)$ and the base case $(0,())$ is unchanged. For $\abs k>0$, the recursion step from $(m,O(k^-))$ to $(m,k)$ is exactly the same as the previous recursion step from $O(k^-)$ to $k$. Thus it remains to handle the step from $(m,())$ to $(m+1,())$.\par
  Write $(X,P)=\sum_i(X_i,P_i)$. Given $\Th_{m+1;()}(I,T^{m;k,P})$, we need to compute the set of possible $\Th_{m;()}(X,P\set x)$ as $x$ varies over $X$. Again, since $\Th_{m;()}(X,P\set x)=\sum_i\Th_{m;()}(X_i,P\set x)$, by recursion it suffices to compute the set of possible $\Th_{m;()}(I,T^{m;(),P\set x})$. We thus describe for a labeling $f$ of $I$ in terms of $\Th_{m+1;()}(I,T^{m;k,P},f)$ whether $f\in\setb{T^{m;(),P\set x}}x$. There is $x\in X$ with $f=T^{m;(),P\set x}$ if and only if there is $i\in I$ such that $f(i)\in T^{m+1;(),P}(i)$ (equivalently, there is $x\in X_i$ with $f(i)\in T^{m;(),P\set x}(i)$) and for $j\ne i$ we have that $f(j)=T^{m;(),P\emptyset}(j)$. This can be described using $\elem$ for the index set.
\end{proof}
In the following, it is convenient to work in a finite partial Morleyization of the language of order. This simplifies the treatment of low quantifier ranks, for example in \cref{compareshelah}. More precisely, for $m\in\N$ let $\Phi_m$ be the set of monadic formulas such that
\begin{itemize}
\item the formulas contain only first-order quantifiers,
\item the formulas are in prenex normal form,
\item existential and universal quantifiers alternate,
\item there are exactly $m$ quantifiers.
\end{itemize}
For instance, $\Th_{m;()}(X,P)=\Th_{m;()}(Y,Q)$ if and only if $(X,P)$ and $(Y,Q)$ satisfy the same formulas in $\Phi_m$. A formula is equivalent to an element of $\bigcup_m\Phi_m$ if and only if it only contains first-order quantifiers. From now on, we work in the monadic language obtained by adding predicates for all elements of some fixed $\Phi_m$ with $m$ defined below. Then $\Th_k$ in this new language is $\Th_{m;k}$ in the old language. Thus, \cref{monadicsum} and its corollaries continue to hold.\par
Let us specify $m$. It is desirable to express any first-order definable order-theoretic or topological property by a formula without monadic quantifiers (for example, being closed, being dense, having isolated points, being bounded on either side, being an open or closed interval, or having an immediate successor or predecessor). We will use this only for one property at a time, hence only for finitely many properties. Fix $m$ large enough such that $\Phi_m$ expresses all these finitely many properties used in the remaining text.
\subsubsection{Labelings}
In our setting, labelings are more convenient to work with than subsets, but quantifying over them has the same expressive power. Thus, we adapt $\Th_k$ to work with $\mc B$-labelings.\par
For a structure $X$ with a $\mc B$-labeling $P\colon X\to I$ and $f\colon J\to I$, we define the possible types of refinements,
\[\refin_f(\Phi)(X,P)=\setb{\Phi(X,PP')}{P'\text{ a }\mc B\text{-labeling with }f\circ P'=P}.\]
For a sequence $k$ of surjections $I_n\xto{k_n}I_{n-1}\xto{k_{n-1}}\dots\xto{k_2}I_1\xto{k_1}I_0$ of finite sets, let
\[\pTh_k=\refin_{k_n}(\refin_{k_{n-1}}(\dots(\refin_{k_1}(\Th_{()})\dots))).\]\par
For such a sequence $k$ of surjections we use the same notation as for tuples of natural numbers. For example, we write $\abs k$ for its length and $k^-$ for the sequence without $k_n$.
\begin{definition}\label{ordquantifiers}
  We define a preorder on the set of sequences of surjections by $k\le k'$ if $\abs k\le\abs{k'}$ and there is a commutative diagram
  \[\begin{tikzcd}
  &&I_{\abs k}\ar{r}{k_{\abs k}}\ar[dashed]{d}{\exists}&I_{\abs k-1}\ar{r}{k_{\abs k-1}}\ar[dashed]{d}{\exists}&\cdots\ar{r}{k_1}&I_0\ar[dashed]{d}{\exists}\\
  I'_{\abs{k'}}\ar{r}{k'_{\abs{k'}}}&\cdots\ar{r}{k'_{\abs k+1}}&I'_{\abs k}\ar{r}{k'_{\abs k}}&I'_{\abs k-1}\ar{r}{k'_{\abs k-1}}&\cdots\ar{r}{k'_1}&I'_0
  \end{tikzcd}\]
  with the vertical maps injective.
\end{definition}
This ordering is analogous to the componentwise ordering on tuples of natural numbers. In particular, if $k\le k'$, then $\pTh_k(X,P)$ can be computed from $\pTh_{k'}(X,P)$ uniformly in $k,k'$. Again, we write $O(k)$ for a large enough uniformly computable sequence of the same length as $k$.
\begin{proposition}
  The maps $\Th_{O(k)}(X,P)\mapsto\pTh_k(X,P)$ and $\pTh_{O(k)}(X,P)\mapsto\Th_k(X,P)$ are well-defined in $X,P$ and computable.
\end{proposition}
The proof generalizes the observations that a subset of $X$ can be coded by its characteristic function, which is a labeling, while a labeling can be coded by its fibers, which are subsets.
\begin{proof}
  By recursion it suffices to show for $P\colon X\to I$ that:
  \begin{enumerate}
  \item Given $f\colon J\to I$ we can compute $m\in\N$ such that from $\asub_m(\Phi)(X,P)$ we can compute $\refin_f(\Phi)(X,P)$ uniformly in $f$.
  \item Given $m\in\N$ we can compute $f\colon J\to I$ such that from $\refin_f(\Phi)(X,P)$ we can compute $\asub_m(\Phi)(X,P)$ uniformly in $I,m$.
  \end{enumerate}
  For (a), check from $\Phi(X,PP')$ if $P'_1,\dots,P'_m$ form the fibers of a refinement of $P$ by a quantifier-free monadic formula. For (b), let $J=I\times\pows{\set{1,\dots,m}}$ and $f\colon J\to I$ be the projection. Then the type after adding subsets $P'_1,\dots,P'_m$ can be detected from the labeling $X\to J,x\mapsto(P(x),\setb k{x\in P'_k})$.
\end{proof}
For the singleton order $\set\star$, for all $k,k'$ the data of $\Th_{()}(\set\star,P),\Th_k(\set\star,P),\pTh_{k'}(\set\star,P)$ are readily computable from each other (they only state the label $P(\star)$). We do not distinguish them in the following, but just write $\Th(\set\star,P)$ for any of them. Similarly, we denote any $\Th_k(\emptyset,P)$ or $\pTh_{k'}(\emptyset,P)$ by $\Th(\emptyset)$.

\subsubsection{Labeling Cantor sets}
From now on, we fix a sufficiently stable Boolean algebra $\mc B$.
\begin{definition}
  Let $I,J$ be finite sets, $X\cong\C$ and $Y\cong\R$.
  \begin{parts}
  \item A labeling $P\colon X\to I\amalg J$ is \textbf{jump-normalized} if it maps left endpoints of jumps into $J$ and all other points into $I$.
  \item Let $C\subseteq Y$ be a Cantor set and $P\colon Y\to I$ be a $\mc B$-labeling. Let $J=\fTh_{O(k)}(\#I)$ and $\exten_Y(P)\colon C\to I\amalg J$ be the jump-normalized labeling that satisfies $\exten_Y(P)(x)=\pTh_k(\rointer{x,y},P)$ for all jumps $(x,y)$ and coincides with $P$ on all other points.
  \end{parts}
\end{definition}
The property of being jump-normalized is quantifier-free definable as being a jump is quantifier-free definable. Our next aim is to reconstruct $\pTh_k(Y,P)$ from the refinement $\exten_Y(P)$ of $\pTh_k(C,P)$.
\begin{proposition}\label{reconexten}
  Let $Y\cong\R$ and $P\colon Y\to I$ be a $\mc B$-labeling and $C\subseteq Y$ be a Cantor set. The maps
  \[\recon\colon\pTh_{O(k)}(C,\exten_Y(P))\mapsto\pTh_k(Y,P)\]
  are well-defined in $P$ and computable.
\end{proposition}
Since $(Y,P)$ is the lexicographic sum $\sum_{x\in C}(I_x,P)$ where $I_x=\rointer{x,y}$ for all jumps $(x,y)$ and $I_z=\set z$ for all other points $z$, for \cref{reconexten} it suffices to show the following lemma.
\begin{lemma}\label{monadicsumcantor}
  Let $X,Y\subseteq\R$ be locally closed and $I_x$ for $x\in X$ be total orders. Suppose that $\setb{I_x}{x\in X}/\cong$ is finite. Let $P$ be a $\mc B$-labeling of $Y$ with $(Y,P)\cong\sum_{x\in X}(I_x,P)$. Set $T^{kP}\colon x\mapsto\pTh_k(I_x,P)$. Then the maps $\pTh_{O(k)}(X,T^{kP})\mapsto\pTh_k(Y,P)$ are well-defined in $X,I_x,P$ and computable.
\end{lemma}
\begin{proof}
  The only difference to \cref{monadicsum} is that the index set $X$ does not admit unrestricted quantification but only quantification over $\mc B$. We show that the same algorithm remains correct.\par
  The base case is unchanged. For the induction step, we use the same notation as in the previous proof. Now $P'$ varies over tuples in $\mc B$ and $f$ varies over $\mc B$-labelings. Thus, for correctness of the previous algorithm it suffices to show that, for a refinement $P'$ of $P$ such that all $P'|_{I_x}$ are $\mc B$-labelings,
  \begin{enumerate}[label=(\roman*)]
  \item if $P'$ is a $\mc B$-labeling, so is $T^{k^-,P'}$;
  \item if $T^{k^-,P'}$ is a $\mc B$-labeling, then there is a $\mc B$-refinement $Q'$ of $P$ with $T^{k^-,P'}=T^{k^-,Q'}$.
  \end{enumerate}\par
  Each $I_x$ is convex in $Y$, thus locally closed in $\R$. By convexity $\#I_x\le1$ for all but countably many $x$. Write $C=\setb x{\#I_x=1}$ and $D=\sum_{x\in C}I_x\subseteq Y$.\par
  For (i), it suffices to show that the restriction of $T^{k^-,P'}$ to $\setb x{\#I_x\le1}\in\mc B$ is a $\mc B$-labeling. Since $Y$ is locally closed, so is $\setb x{I_x\ne\emptyset}$. (Indeed, working locally, suppose that $Y$ is closed. Since $Y$ is Dedekind complete, $\setb x{I_x\ne\emptyset}$ is Dedekind complete, thus locally closed.) Since $(C,T^{k^-,P'})\cong(D,P')$ (identifying labels $\Th(\set x,P')$ with $P'(x)$), also $T^{k^-,P'}|_C$ is a $\mc B$-labeling.\par
  For (ii), for each type $s$ of some $I_x$ and isomorphism class $t$ of $I_x$ as unlabeled order with $\#I_x>1$, choose a realization $Q'_{st}$ of $s$ and $t$. Define $Q'$ by $Q'|_D=P'|_D$ and $Q'|_{I_x}=Q'_{\pTh_{k^-}(I_x,P'),I_x}$. Then $T^{k^-,P'}=T^{k^-,Q'}$ by construction and it remains to show $Q'$ a $\mc B$-labeling.\par
  Since $(D,Q')=(D,P')\cong(C,T^{k^-,P'})$, the restriction $Q'|_D$ is a $\mc B$-labeling. To show that $Q'|_{Y\setminus D}$ is a $\mc B$-labeling, let $\mc A$ be the generator of $\mc B$ in \cref{adequateall}. First, since there are finitely many $Q'_{st}$, there is a finite bound on the number of elements of $\mc A$ used in some $Q'_{st}$. Second, by the defining property of $\mc A$ for the partition into $D$ and the $I_x$, taking an element of $\mc A_{I_x}$ for each $x$ yields an element of $\mc A_Y$. We conclude that each fiber of $Q'|_{Y\setminus D}$ is a finite Boolean combination of elements of $\mc A$.
\end{proof}
\begin{definition}
  \begin{parts}
  \item Let $X\subseteq\R$ be locally closed and $P\colon X\to I$ be a $\mc B$-labeling. We define
    \[\cant(\Phi)(X,P)=\set{\Phi(C,\exten_X(P))\mid C\subseteq X\text{ a Cantor set}}.\]
    If $t\in\cant(\Phi)(X,P)$, we often say that $t$ occurs in $(X,P)$.
  \item Let $C\cong\C$ and $P\colon C\to I$ be a $\mc B$-labeling. We obtain a labeling $\R_C(P)\colon R\to I\amalg I\times I$ with $R\cong\R$ as in \cref{cantorquot} by leaving all points of $C$ that are no component of a jump unchanged and replacing all jumps $(a,b)$ by a point of label $(P(a),P(b))$.
  \end{parts}
\end{definition}
By \cref{monadicsumcantor}, $\pTh_{O(k)}(C,P)$ and $\pTh_{O(k)}(R,\R_C(P))$ are computable from each other. Indeed, there are decompositions $C\cong\sum_{r\in R}I_r$ with $\#I_r\in\set{1,2}$ and $R\cong\sum_{c\in C}J_c$ with $\#J_c\in\set{0,1}$.

\subsection{Uniform labelings}
Our next aim is to show that over the reals, we can reduce to quantifying over Cantor sets and quantifying over sets satisfying a certain uniformity condition (\cref{shelahres}).
\begin{proposition}\cite[Lemma 5.3]{shelahord}
  Let $X\subseteq\R$ be locally closed, $P$ be a $\mc B$-labeling of $X$ and $k$ be a sequence of surjections. The following conditions are equivalent
  \begin{tfae}
  \item all nonempty relatively convex $I,J\subseteq X$ without endpoints satisfy $\pTh_k(I,P)=\pTh_k(J,P)$;
  \item there are a dense $D\subseteq X$ and types $p,q$ such that all $a,b\in D$ satisfy $\pTh_k(\ointer{a,b},P)=p$ and $\Th(\set a,P)=q$.
  \end{tfae}
\end{proposition}
\begin{proof}
  Assuming (i), let $D$ be a nonempty fiber of $P$. Since every nonempty relatively convex subset has the same type as $X$ and this type guarantees an element of $D$ in any realization, the set $D$ is dense.\par
  Suppose that (ii) holds and let $I,J$ be nonempty convex without endpoints. Since $D$ is dense, there are strictly increasing sequences $(a_n)_{n\in\Z},(b_n)_{n\in\Z}$ in $D$ with $a_{-n}\to\inf I,a_n\to\sup I,b_{-n}\to\inf J,b_n\to\sup J$ as $n\to\infty$. Thus
  \begin{multline*}
    \pTh_k(I,P)=\sum_{n\in\Z}\Th(\set{a_n},P)+\pTh_k(\ointer{a_n,a_{n+1}},P)=\sum_{n\in\Z}(q+p)\\
    =\sum_{n\in\Z}\Th(\set{b_n},P)+\pTh_k(\ointer{b_n,b_{n+1}},P)=\pTh_k(J,P).
  \end{multline*}
\end{proof}
We call $P$ $k$-\textbf{uniform} if these equivalent conditions hold.\par
A $\mc B$-labeling $P$ of $\C$ is $k$-uniform if and only if $\R_\C(P)$ is $k'$-uniform for $k'$ the image of $k$ under the endofunctor $I\mapsto I\amalg I\times I$ of finite sets. Moreover, iso-uniform $\mc B$-labelings are $k$-uniform for all $k$.
\begin{lemma}\label{somewhereunif}
  Let $P$ be a $\mc B$-labeling of $X$ with $X\cong\R$ or $X\cong\C$ and let $k$ be a sequence of surjections. Then there is an interval $I\subseteq X$ that is $k$-uniformly labeled by $P$.
\end{lemma}
\begin{proof}
  By \cref{dloramsey} or \cref{cantorramsey} there are an interval $I$ and a dense $D\subseteq I$ that is homogeneous for the additive coloring $(a,b)\mapsto\pare{\pTh_k(\ointer{a,b},P),\Th(\set a,P)}$.
\end{proof}
\begin{remark}\label{extenunif}
  If $P$ is $k$-uniform and $C\subseteq\R$ is a Cantor set, then given $\pTh_{O(k)}(C,P)$, or equivalently $\pTh_{O(k)}(\R,\R_C(P))$, together with $\pTh_k(\R,P)$, we can compute $\pTh_k(C,\exten_\R(P))$. Indeed, every complementary interval has type $\pTh_k(\R,P)$.
\end{remark}
Let $\Phi$ be an element of some $\fTh_k(\#P)$ for $P$ a $\mc B$-labeling of $X$ with $X\cong\R$ or $X\cong\C$. We define $\urefin_f^k(\Phi)(X,P)$ as the set of $\Phi(X,PP')$ for $P'$ ranging over all jump-normalized $k$-uniform $\mc B$-refinements of $P$ along $f$.\par
We define a uniform analogue of $\cant$, called $\ucant^k(\Phi)(X,P)$, as the set of $\Phi(C,P)\in\cant(\Phi)(X,P)$ for $k$-uniformly labeled $C$. Note that $\urefin_f^k(\ucant^k(\Phi))=\urefin_f^k(\cant(\Phi))$ as coarsening a $k$-uniform $\mc B$-labeling yields a $k$-uniform $\mc B$-labeling. We call this composite also $\urefincant_f^k(\Phi)$.\par
Now suppose that $P$ is $k$-uniform. We define $\uTh_k(X,P)$ by recursion on $\abs k$. Set $\uTh_{()}(X,P)=\Th_{()}(X,P)$. If $\abs k>0$, then $\uTh_k(X,P)=(\uTh_k^1(X,P),\uTh_k^2(X,P))$ with $\uTh_k^1(X,P)=\urefin_{k_n}^{k^-}(\uTh_{k^-})(X,P)$ and $\uTh_k^2(X,P)=\urefincant_{k_n}^{k^-}(\uTh_{k^-})(X,P)$.\par
\begin{remark}\label{compareshelah}
  For $\uTh_k^2$, \cite[Definition 5.2]{shelahord} used convex equivalence relations instead of Cantor sets. This also allows handling incomplete total orders, but is equivalent for the reals. We use Cantor sets for the more direct applicability of Baire category methods. He worked with $\Th$ instead of $\pTh$ (although he also used an analogue of $\pTh$ elsewhere), and the difference is irrelevant. Due to Morleyizing the language, we obtain that $\pTh_{O(k)}(\R,P)\mapsto\uTh_k(\R,P)$ is well-defined and computable, while he needed to increase $\abs k$ for this \cite[Lemma 5.1]{shelahord}.
\end{remark}
\begin{theorem}\cite[Theorem 5.4]{shelahord}\label{shelahres}
  The maps $\pTh_{O(k)}(\R,P)\mapsto\uTh_k(\R,P)$ and $\uTh_{O(k)}(\R,P)\to\pTh_k(\R,P)$ are well-defined in $k$-uniform $P$ and computable.
\end{theorem}
\begin{proof}
  Computing $\refin(\cant)$ from $\refin$ is easy using \cref{subcantoriff}. Thus, for computing $\pTh_{O(k)}(\R,P)\mapsto\uTh_k(\R,P)$ it suffices to decide, given $\pTh_{O(k)}(\R,P)$, whether $P$ is $k^-$-uniform. This is easy by quantifying over intervals.\par
  For the converse computation, let $P\colon\R\to I$. By recursion it suffices to compute $\pTh_k$ from $\urefin^{k^-}_f(\pTh_{k^-})$ and $\urefincant^{k^-}_f(\pTh_{k^-})$ for a large enough (in the sense of \cref{ordquantifiers}) map $f$ onto $I$. Informally, we show that each element of $\pTh_k(\R,P)$ is obtained from uniform elements (in (i) below) by repeated lexicographic addition ((ii) and (iii) below), and by an operation (iv) that combines the given types of uniformly labeled Cantor sets and previously obtained types of complementary intervals. Formally, define $E$ as the smallest subset of $\fTh_{k^-}(\#P)$ such that the following conditions hold:
  \begin{enumerate}[label=(\roman*)]
  \item $\urefin^{k^-}_{k_n}(\pTh_{k^-})(\R,P)\subseteq E$.
  \item Let $t_1,t_2\in E$ be arbitrary and $t=\Th(\set x,Q)$ for some $x\in\R$ and refinement $Q$ of $P$. Then $t_1+t+t_2\in E$.
  \item Let $t_1\in E$ be arbitrary and $t=\Th(\set x,Q)$ for some $x\in\R$ and refinement $Q$ of $P$. Then $\sum_{i\in\N}(t_1+t)\in E$ and $\sum_{i\in\dual{\N}}(t+t_1)\in E$.
  \item Define $J=\fTh_{O(k^-)}(\#I)$ and let $g\colon I\amalg J\to I$ be the identity on $I$ and such that $g(\pTh_{O(k^-)}(\rointer{x,y},P))=P(x)$. Let $h\colon J\to J$ be such that $h(\pTh_{O(k^-)}(\rointer{x,y},P))=\pTh_{O(k^-)}(\ointer{x,y},P)$. Consider an arbitrary $\pTh_{k^-}(\C,Q)\in\urefincant_g^{k^-}(\pTh_{k^-})(\R,P)$ with $h(Q(\C)\cap J)\subseteq E$. Then $\recon(\pTh_{k^-}(\C,Q))\in E$.
  \end{enumerate}
  Since $E$ is clearly computable, it suffices to show that $E=\pTh_k(\R,P)$.\par
  For showing $E\subseteq\pTh_k(\R,P)$, we use induction on the construction. The base case (i) follows from
  \[\urefin^{k^-}_{k_n}(\pTh_{k^-})(\R,P)\subseteq\refin_{k_n}(\pTh_{k^-})(\R,P)=\pTh_k(\R,P).\]
  For (ii), suppose that $t_1,t_2\in\pTh_k(\R,P)$ and $t=\Th(\set x,Q)$ for some refinement $Q$ of $P$. The $k$-uniformity of $P$ yields $t_1\in\pTh_k(\R,P)=\pTh_k(\ointer{-\infty,x},P)$ and $t_2\in\pTh_k(\R,P)=\pTh_k(\ointer{x,\infty},P)$, say $Q_1,Q_2$ are $\mc B$-refinements of $P$ with $t_1=\pTh_{k^-}(\ointer{-\infty,x},Q_1)$ and $t_2=\pTh_{k^-}(\ointer{x,\infty},Q_2)$. Then $Q_1,Q_2,Q$ combine to $Q'$ with
  \[t_1+t+t_2=\pTh_{k^-}(\ointer{-\infty,x},Q_1)+\Th(\set x,Q)+\pTh_{k^-}(\ointer{x,\infty},Q_2)=\pTh_{k^-}(\R,Q')\in\pTh_k(\R,P).\]
  Similarly, splitting $\R$ into countably many intervals separated by points we obtain (iii), and splitting $\R$ into intervals separated by a Cantor set we obtain (iv).\par
  For the converse inclusion, let $\pTh_{k^-}(\R,Q)\in\pTh_k(\R,P)$ be arbitrary. We need the following result.
  \begin{sublemma}\label{lemshelahres}
    Let $Q$ be a $\mc B$-labeling of $\R$ and $I\subseteq\R$ be convex without endpoints such that for all $x\in I$ there are $a,b\in I$ with $a<x<b$ such that all $a',b'$ with $a<a'<b'<b$ satisfy $\pTh_{k^-}(\ointer{a',b'},Q)\in E$. Then $\pTh_{k^-}(I,Q)\in E$.
  \end{sublemma}
  \begin{proof}
    By assumption we can cover $I$ with open intervals $J$ such that for all subintervals $K\subseteq J$ without endpoints we have that $\pTh_{k^-}(K,Q)\in E$. By compactness we may assume that every closed subinterval of $I$ meets finitely many $J$. Writing $I$ as an increasing union of such closed intervals, we obtain a strictly increasing sequence $(a_n)_{n\in\Z}$ with $a_{-n}\to\inf I,a_n\to\sup I$ as $n\to\infty$ and $\pTh_{k^-}(\ointer{a_n,a_{n+1}},Q)\in E$ for all $n\in\Z$. Using (ii), we obtain $\pTh_{k^-}(\ointer{a_m,a_n},Q)\in E$ for all integers $m<n$. Apply \cref{ramsey} to the colorings $(m,n)\mapsto\pare{\pTh_{k^-}(\ointer{a_m,a_n},Q),\Th(\set{a_m},Q),\Th(\set{a_n},Q)}$ of $\N$ and $-\N$ separately. Replacing $(a_n)_{n\in\Z}$ by the subsequences thus obtained we may assume that there are $s,t,s_1,t_1$ with $\Th(\set{a_n},Q)=s$ for $n<0$, $\Th(\set{a_n},Q)=t$ for $n>0$, $\pTh_{k^-}(\ointer{a_{n-1},a_n},Q)=s_1$ for $n<0$, and $\pTh_{k^-}(\ointer{a_n,a_{n+1}},Q)=t_1$ for $n>0$. By (ii) and (iii)
  \[\pTh_{k^-}(I,Q)=\sum_{i\in\dual\N}(s+s_1)+\Th(\set{a_{-1}},Q)+\pTh_{k^-}(\ointer{a_{-1},a_1},Q)+\Th(\set{a_1},Q)+\sum_{i\in\N}(t_1+t)\in E.\qedhere\]
  \end{proof}
  Let $C\subseteq\R$ be the set of $x$ such that for all $a,b$ with $a<x<b$ there are $a',b'$ with $a<a'<b'<b$ and $\pTh_{k^-}(\ointer{a',b'},Q)\notin E$. Clearly $C$ is closed.\par
  We show that $C$ has no isolated points. Suppose for the sake of contradiction that $C\cap\ointer{a,b}=\set x$. By \cref{lemshelahres} all subintervals $I$ of $\ointer{a,x},\ointer{x,b}$ satisfy $\pTh_{k^-}(I,Q)\in E$. Since $x\in C$, there are $a',b'$ with $a<a'<b'<b$ and $\pTh_{k^-}(\ointer{a',b'},Q)\notin E$. Then $a'<x<b'$ as otherwise we may set $I=\ointer{a',b'}$, and similarly $\pTh_{k^-}(\ointer{a',x}),\pTh_{k^-}(\ointer{x,b'})\in E$. This contradicts (ii).\par
  We show that $C$ has empty interior. Let $I$ be any interval. By \cref{somewhereunif} there is an open subinterval $J\subseteq I$ such that $Q$ is $k$-uniform on $J$. Thus for $a',b'\in J$ with $a'<b'$ by (i)
  \[\pTh_{k^-}(\ointer{a',b'},Q)\in\urefin_{k_n}^{k^-}(\pTh_{k^-})(\R,P)\subseteq E.\]
  In other words, $J\cap C=\emptyset$. Hence $I\nsubseteq C$.\par
  Applying \cref{subcantoriff,cantorderisom} in $\R\cup\set{-\infty,\infty}$ we conclude that $C$ is order isomorphic to one of $\emptyset,\C,\tsingl+\C,\C+\tsingl,\overline\C$. If $C$ is empty, then \cref{lemshelahres} with $I=\R$ shows $\pTh_{k^-}(\R,Q)\in E$.\par
  Else $C$ includes a subset $D\cong\C$ that is convex in $C$. Shrinking $D$, by \cref{somewhereunif} we may assume that $Q$ is $k$-uniform on $D$. Let $I=\ointer{\inf D,\sup D}$ be the convex hull of $D$ in $\R$ and $J\subseteq I$ be any subinterval. Then by definition of $C$ and \cref{lemshelahres}, the type of each complementary interval of $D$ in $J$ lies in $E$, whence (iv) yields $\pTh_{k^-}(J,Q)\in E$. This shows $D\subseteq I\cap C=\emptyset$, which is absurd.
\end{proof}
For later use, we record a consequence of the proof.
\begin{lemma}\label{shelahresdecomp}
  Let $P$ be a $\mc B$-labeling of $\R$ and $k$ be a sequence of surjections. Then there is a realization $P'$ of $\Th_k(\R,P)$ such that $\R$ can be decomposed as a countable disjoint union of open intervals, Cantor sets and points on which $P'$ is $k$-uniform.
\end{lemma}
\begin{proof}
  Let $E$ be as in the proof of \cref{shelahres}. Since $\Th_k(\R,P)\in E$, we obtain a realization $P'$ such that $(\R,P')$ is constructed by a finite number of applications of the operations of (i) to (iv). Each of these operations preserves the property to admit such a decomposition.
\end{proof}
\begin{example}\cite[Theorem 6.5]{shelahord}
  Consider the Boolean algebra generated by the countable sets as $\mc B$. This requires minor modifications since $\mc B$ is not sufficiently stable and the proof of \cref{monadicsumcantor} fails. However, the reader may check that, using the notation of the proof, the subclass of the $T^{k^-,PP'}$ such that there is a $\mc B$-labeling $Q'$ with $T^{k^-,PP'}=T^{k^-,PQ'}$ is computable. Thus \cref{monadicsumcantor} persists.\par
  We prove decidability of the restricted monadic theory. Informally, we obtain full quantifier elimination for uniform sets. More precisely, by \cref{shelahres} it suffices to show the following. Let $X_1,\dots,X_n\in\mc B$ be the fibers of a valid $()$-uniform $\mc B$-labeling. In other words, $\R$ is the disjoint union of the $X_i$ and each $X_i$ is empty or dense. Then the labeling is $k$-uniform for all $k$ and the maps $\uTh_{()}(\R,X_1,\dots,X_n)\mapsto\uTh_k(\R,X_1,\dots,X_n)$ are well-defined in the $X_i$ and computable uniformly in $k$. By recursion, it suffices to compute $\uTh_{(k_1)}(\R,X_1,\dots,X_n)$.\par
  For $\uTh^1_{(k_1)}$ with $k_1\colon\set{1,\dots,m}\to\set{1,\dots,n}$, as in \cref{exampledlo} we show that $\uTh^1_{(k_1)}(\R,X_1,\dots,X_n)$ is the set of all $\uTh_{()}(\R,Y_1,\dots,Y_m)$ with $Y_i\subseteq X_{k_1(i)}$ for all $i$. Indeed, $\uTh_{()}(\R,X_1,\dots,X_n)$ only contains the information which of the $X_i$ is the unique cocountable one and which ones are empty. Thus the statement follows from the following two facts.
  \begin{itemize}
  \item If $X\subseteq\R$ is dense, then there are two disjoint dense subsets of $X$.
  \item If $X\subseteq\R$ is cocountable, then there are two disjoint dense subsets of $X$ one of which is countable.
  \end{itemize}
  Thus, for computing $\uTh^2_{(k_1)}$ it suffices to compute $\ucant$. Similarly to the previous case, we show that a type $t$ of a uniformly labeled Cantor set occurs in $(\R,X_1,\dots,X_n)$ whenever there are $X_i'$ with $\uTh_{()}(\R,X_1',\dots,X_n')=\uTh_{()}(\R,X_1,\dots,X_n)$ and such that $t$ occurs in $(\R,X_1',\dots,X_n')$. A partial version is the following statement.
  \begin{itemize}
  \item If $A_1,\dots,A_m,B_1,\dots,B_n\subseteq\R$ are countable, dense and pairwise disjoint, then there is a Cantor set $C$ such that all $A_j$ are disjoint from $C$ and all $B_i$ are dense in $C$.
  \end{itemize}
  The general statement also considers endpoints of jumps and follows easily from \cref{findcantor}.\par
  \cite[Theorem 6.2]{shelahord} used a variant of this argument to compute the common unrestricted monadic theory of the rationals and all Specker orders (also called Aronszajn lines).
\end{example}

\section{Nesting Cantor sets}\label{secdecproofa}
From now on, we assume that all members of the sufficiently stable Boolean algebra $\mc B$ have the Baire property. This is a strictly weaker assumption than determinacy.\par
We define a notion of coarse types, which contain a limited amount of data that nonetheless determines the type of a uniform $\mc B$-labeling. Essentially, we may forget all data about refinements, $\urefin$, and only need to remember the types of Cantor subsets, $\ucant$.\par
In this section, we define coarse types and compute which ones are satisfiable. In the coming \cref{secdecproofb} we show how to compute uniform refinements of a coarse type.

\subsection{Expanding the language}\label{secexpandlang}
To eliminate more quantifiers, we now expand the Morleyized monadic language of \cref{secfoquant} by a predicate for meager sets, thus changing the meaning of types like $\pTh_k$. The results of the previous sections, in particular \cref{shelahres}, persist in the expanded language, as argued here.\par
The same proofs apply if we obtain analogues of \cref{monadicsum} for countable index sets and \cref{monadicsumcantor}. Since the induction steps are handled as before, it suffices to consider the base case $\abs k=0$. We distinguish three cases:
\begin{enumerate}[label=(\roman*)]
\item For \cref{monadicsum}, the index set is countable.
\item For the first application of \cref{monadicsumcantor} in \cref{reconexten}, the union of all summands that are singletons is a Cantor set in $\R$, thus meager.
\item For the remaining applications of \cref{monadicsumcantor} in computing $\pTh_{O(k)}(\C,P)$ and $\pTh_{O(k)}(\R,\R_\C(P))$ from each other, all summands are finite.
\end{enumerate}
In case (i) or (ii), a subset of the sum is meager if and only if it is meager in all infinite summands. In case (iii), a subset $A$ is meager if and only if the set of summands $I_x$ with $A\cap I_x\ne\emptyset$ is meager.

\subsection{Coarse types}
\begin{definition}
  Let $I$ be a finite set of labels and $n\in\N$.
  \begin{parts}
  \item We define $\fcTh_n(I)$ by recursion on $n$. Set $\fcTh_0(I)=I\times\pows I$, and for $(i,J)\in\fcTh_0(I)$ call $i$ the \textbf{comeager label}, elements of $J$ \textbf{realized labels}, and elements of $I\setminus J$ \textbf{omitted labels}. Set $\fcTh_{n+1}(I)=\fcTh_0(I)\times\pows{\fcTh_n(I\amalg I\times I)}$, with the first component called \textbf{underlying $0$-type} and elements of the second component called \textbf{Cantor} $n$-\textbf{subtypes}.
  \item We define whether a $\mc B$-labeling $P$ of $R\cong\R$ or $C\cong\C$ by $I$ \textbf{realizes} $t\in\fcTh_n(I)$ by recursion on $n$, and whether it is \textbf{coarsely $n$-uniform}.
    \begin{itemize}
    \item A $\mc B$-labeling of $R$ realizes a coarse $0$-type $(i,J)$ if
      \begin{itemize}
      \item the fiber of the comeager label $i$ is indeed comeager;
      \item the fibers of the realized labels $j\in J$ are dense; and
      \item the fibers of the omitted labels $j'\notin J$ are empty.
      \end{itemize}
    \item Given the notion of realizing coarse $n$-types for $\mc B$-labelings of $\R$, we say that
      \begin{itemize}
      \item a $\mc B$-labeling $Q$ of $C$ with $m$ labels realizes a coarse $n$-type $t\in\fcTh_n(I\amalg I\times I)$ if $\R_C(Q)$ realizes $t$; and
      \item a $\mc B$-labeling of $R\cong\R$ or $C\cong\C$ is coarsely $n$-uniform if it realizes some $\fcTh_n(I)$.
      \end{itemize}
    \item A $\mc B$-labeling $P$ of $R$ realizes a coarse $(n+1)$-type $(t_0,T)$ if
      \begin{itemize}
      \item $P$ realizes the underlying $0$-type $t_0$;
      \item for each Cantor subtype $t\in T$, every nonempty open subset of $R$ has a Cantor subset that realizes $t$; and
      \item each coarsely $n$-uniformly labeled Cantor subset of $(R,P)$ realizes some Cantor subtype $t\in T$.
      \end{itemize}
    \end{itemize}
  \item For a coarsely $n$-uniform $\mc B$-labeling $P$ of $R\cong\R$ or $C\cong\C$, write $\cTh_n(R,P)$ (respectively $\cTh_n(C,P)$) for the unique element of $\fcTh_n(I)$ (respectively $\fcTh_n(I\amalg I\times I)$) realized by $P$, and call it the \textbf{coarse type} of $P$.
  \end{parts}
\end{definition}
As for usual types, we use ``realizes'' and ``satisfies'' synonymously for coarse types. We sometimes use the terminology for coarse types for arbitrary labelings, calling labels \textbf{realized} if their fiber is nonempty and \textbf{omitted} else, and referring to types of coarsely $n$-uniformly labeled Cantor sets of a labeled order $X$ as the \textbf{Cantor} $n$-\textbf{subtypes} of $X$.\par
\begin{example}
  We give an equivalent description of coarse uniformity and coarse $n$-types that is more directly analogous to $\uTh_k$.\par
  Given a $\mc B$-labeling $P\colon\R\to I$ and $n\in\N$, we define a preliminary $n$-theory $T_n(P)$. For $n=0$, let $T_0(P)$ be the pair of the set of labels with empty fiber and the set of labels with meager fiber. These data are encoded in $\uTh_{()}(\R,P)=\Th_{()}(\R,P)$, and for $()$-uniform $P$ they determine $\Th_{()}(\R,P)$ completely. Let $T_{n+1}(P)$ be the set of Cantor $n$-subtypes of $(\R,P)$. In other words, $T_{n+1}(P)$ is the set of $\cTh_n(\R,\R_C(P))$ as $C$ varies over $n$-uniformly labeled Cantor subsets.\par
  Then $P$ is coarsely $n$-uniform if and only if $T_n(P)=T_n(P|_J)$ and $T_0(P)=T_0(P|_J)$ for all nonempty open intervals $J$. For $n=0$, although there is no formal equality, the data encoded by $\cTh_0(\R,P)$ and $T_0(P)$ coincide for coarsely $0$-uniform $P$ since by the Baire property there is a unique comeager label. For $n>0$, we have that $\cTh_n(\R,P)=(\cTh_0(\R,P),T_n(P))$ for coarsely $n$-uniform $P$.
\end{example}
We will investigate which coarse types are satisfiable (by a $\mc B$-labeling of $\R$), leading to \cref{satisconstr}.\par
Here we can already answer the question of satisfiability for $0$-types: By a straightforward argument, $(i,J)\in\fcTh_0(I)$ is satisfiable if and only if $i\in J$. For a coarse type with labels in $I\amalg I\times I$ to be satisfiable by a $\mc B$-labeling of $\C$, we need an additional condition: The labels in $I\times I$ of the jumps must contain at least one realized label, but must not contain the comeager label. For $0$-types, it is straightforward that these conditions suffice. Also for general $n$, we call coarse $n$-types satisfiable by a $\mc B$-labeling of $\C$ \textbf{Cantor-satisfiable}. Cantor-satisfiable types are also satisfiable by a $\mc B$-labeling of $\R$.
\begin{lemma}\label{relcoarunif}
  $O(k)$-uniform $\mc B$-labelings are coarsely $\abs k$-uniform.
\end{lemma}
\begin{proof}
  Since coarse $\abs k$-uniformity means that every label or recursively Cantor subtype that occurs in some open subset occurs in every open subset, this follows from the fact that usual types $\Th_{O(k)}(\R,P)$ determine coarse types $\cTh_{\abs k}(\R,P)$.
\end{proof}
\begin{lemma}\label{trivunif}
  \begin{parts}
  \item Open intervals in coarsely $n$-uniformly labeled sets are coarsely $n$-uniformly labeled of the same type.
  \item For all $n$, every labeled set has a coarsely $n$-uniformly labeled subinterval.
  \end{parts}
\end{lemma}
\begin{proof}
  (a) is clear and (b) follows from \cref{somewhereunif,relcoarunif}.
\end{proof}
For the remaining document, we abbreviate ``coarsely uniform'' as ``uniform'' and will not use the previous notion of uniformity. Similarly, we will often abbreviate ``coarse type'' as ``type'' and reserve the unqualified use of ``type'' for coarse types.
\begin{corollary}
  For $n\ge k$, the map $\cTh_n(\R,P)\mapsto\cTh_k(\R,P)$ is well-defined in $P$ and computable uniformly in $n,k$. We denote this map also by $t\mapsto t|k$. If the Cantor subtypes of $t$ are $t_1,\dots,t_m$, then the Cantor subtypes of $t|k$ are $t_1|(k-1),\dots,t_m|(k-1)$.
\end{corollary}
\begin{proof}
  By recursion, it suffices to consider $\cTh_{n+1}(\R,P)\mapsto\cTh_n(\R,P)$. The case $n=0$ holds by definition of $\cTh_1$. For the induction step it suffices to show that if the Cantor subtypes of an $(n+1)$-type $t$ are $t_1,\dots,t_m$, then the Cantor subtypes of $t|n$ are $t_1|(n-1),\dots,t_m|(n-1)$. Clearly each $t_i|(n-1)$ is realized. The converse follows as by \cref{trivunif} each Cantor subtype of $t|n$ also has an $n$-uniform realization, which realizes some $t_i$.
\end{proof}

\subsection{Uniform sums}
We first observe that iso-uniform sums make sense for $\mc B$-labeled orders. More precisely, if $\mc B$ is a sufficiently stable Boolean algebra and $T$ is a finite set of $\mc B$-labeled order types, then the iso-uniform sum of $T$ is a $\mc B$-labeled order type by \cref{adequatepart}. The notion of uniform sum defined below describes the coarse type of an iso-uniform sum (this will follow from \cref{unifcomp}), but is more flexible.\par
Let $n\in\N$ and $T\subseteq\fcTh_n(I\amalg I\times I)$.
\begin{definition}
  A labeled total order $C$ is \textbf{compatible with} $T$ if every Cantor $n$-subtype of $C$ is realized by a Cantor subset of some realization of some $t\in T$.
\end{definition}
In particular, realizations of elements of $T$ are compatible with $T$, and singleton orders are compatible with all $T$. Moreover, subsets of sets compatible with $T$ are compatible with $T$.
\begin{proposition}\label{unioncompat}
  Let $(X,P)$ be a labeled total order and $(C_i)_{i\in\N}$ be a sequence of closed subsets of $X$ compatible with $T$. Then $\bigcup_iC_i$ is compatible with $T$.
\end{proposition}
\begin{proof}
  Let $D\subseteq\bigcup_iC_i$ be an $n$-uniformly labeled Cantor subset. Then some $C_i$ is nonmeager in $D$, thus by closedness has nonempty interior. By uniformity $C_i$ includes a realization of $\cTh_n(D,P)$. Since $C_i$ is compatible with $T$, we conclude that $\cTh_n(D,P)$ can be embedded in some realization of some $t\in T$.
\end{proof}
\begin{definition}\label{defunifsum}
  Let $P$ be a $\mc B$-labeling of $X\cong\R$ or $X\cong\C$. We call $(X,P)$ an $n$-\textbf{uniform sum} of $T$ if
  \begin{itemize}
  \item every nonempty open subset of $X$ includes realizations of all $t\in T$ by meager Cantor sets, and
  \item there is a meager $F_\sigma$ set $M\subseteq X$ compatible with $T$ such that $X\setminus M$ has a single realized label.
  \end{itemize}
\end{definition}
The definition for $X\cong\C$ is equivalent to $(\R,\R_\C(P))$ being an $n$-uniform sum of $T$, using for (a) the bijection of meager Cantor subsets of \cref{projcantorbij}. Thus, when proving statements about uniform sums $(X,P)$ and coarse types, it usually suffices to consider the case $X=\R$.\par
Given realizations of all summands, an iso-uniform sum of these realizations is an $n$-uniform sum. Coarsening an (iso- or $n$-)uniform sum along a map of labels yields a uniform sum of the set of coarsened types.\par
By \cref{trivunif}, whenever $C$ is compatible with a set $\set{t_1,\dots,t_m}$ of $n$-types, then $C$ is also compatible with $\set{t_1|k,\dots,t_m|k}$ for $k\le n$. In particular, any $n$-uniform sum of $\set{t_1,\dots,t_m}$ is a $k$-uniform sum of $\set{t_1|k,\dots,t_m|k}$.
\begin{lemma}\label{unifsumdisj}
  Let $P$ be a $\mc B$-labeling of $\R$. Then $(\R,P)$ is an $n$-uniform sum of $T$ if and only if there are countably many Cantor subsets $C_j$ with pairwise disjoint closures such that every $C_j$ is compatible with $T$, for all $t\in T$ every nonempty open subset includes a $C_j$ realizing $t$, and $\R\setminus\bigcup_jC_j$ has a cocountable label.
\end{lemma}
\begin{proof}
  For sufficiency, let $M$ be the union of all $\overline{C_j}$'s and countable labels of $\R\setminus\bigcup_jC_j$. Then $M$ is compatible with $T$ by \cref{unioncompat}.\par
  For necessity, choose for all nonempty bounded basic open $U_k$ and $t\in T$ a $D_{kt}\subseteq U_k$ realizing $t$. Let $(E_i)_{i\in\N}$ be a sequence of compact, nowhere dense sets that are compatible with $T$ such that $\R\setminus\bigcup_iE_i$ has a cocountable label. Removing isolated points, we may assume $E_i\cong\overline\C$. \Cref{refincantordense} for $I=T$ and $\setb{\overline{D_{kt}}}{k,t}\cup\setb{E_i}{i}$ yields a sequence of pairwise disjoint copies of $\overline\C$ compatible with $T$ such that for all $t\in T$ every nonempty open subset includes such a copy of the form $\overline{D_{kt}}$. Removing endpoints from these copies yields the desired sequence $(C_j)_j$ of Cantor sets.
\end{proof}
\begin{lemma}\label{isunifsum}
  If $P$ is $(n+1)$-uniform, then $(\R,P)$ is an $n$-uniform sum of the set $T$ of Cantor $n$-subtypes of $(\R,P)$.
\end{lemma}
\begin{proof}
  Every subset is compatible with $T$. Thus, we may choose $M$ to be any meager $F_\sigma$ set that includes the complement of the comeager label. By $(n+1)$-uniformity, every element of $T$ occurs in every nonempty open subset.
\end{proof}
\begin{corollary}\label{unifmax}
  Every satisfiable $n$-type $s$ is realized by an $n$-uniform sum of some set $T$ of satisfiable $n$-types. Moreover, for $n>0$ we may choose $T$ such that $\setb{t|(n-1)}{t\in T}$ is the set of Cantor subtypes of $s$.
\end{corollary}
\begin{proof}
  By \cref{trivunif}, $s$ has an $(n+1)$-uniform realization $P$. By \cref{isunifsum}, $(\R,P)$ is an $n$-uniform sum of Cantor its $n$-subtypes. Their $(n-1)$-types are the Cantor subtypes of $s$.
\end{proof}

\subsection{Computing uniform sums}
The aim of this subsection is to prove computability of the coarse type of a uniform sum from the coarse types of its summands (\cref{unifcomp}). Making the computation compatible with our choice to regard labelings $P\colon\C\to I$ as labelings $\R_\C(P)\colon\R\to I\amalg I\times I$ requires a fine grained analysis of the endpoints of jumps in nested Cantor sets.\par
Let $P\colon\C\to I$ be an $(n+1)$-uniform $\mc B$-labeling and $D\subseteq\C$ be an $n$-uniformly labeled meager Cantor subset. For $\pi$ as in \cref{cantorquot}, consider the following commutative diagram:
\[\begin{tikzcd}
  \C\ar{d}{\pi_\C}&\ar[symbol=\supseteq]{l}{}D\ar{d}{\pi_\C}\ar{r}{\pi_D}&\R\ar[equal]{d}\\
  \R&\ar[symbol=\supseteq]{l}{}\pi_\C(D)\ar{r}{\pi_{\pi_\C(D)}}&\R
\end{tikzcd}\]
Then $D$ induces two labelings of the reals. One of them, let us call it $Q_1(P,D)$, with labels in $I\amalg I\times I$, is $\R_D(P|_D)$ for the labeling of $D$ with $m$ labels obtained as a subset of $\C$. The other one, let us call it $Q_2(P,D)$, with labels in $(I\amalg I\times I)\amalg(I\amalg I\times I)\times(I\amalg I\times I)$, is $\R_{\pi_\C(D)}(\R_\C(P)|_{\pi_\C(D)})$ for the labeling of $\pi_\C(D)$ with labels in $I\amalg I\times I$ obtained as a restriction of $\R_\C(P)$ in the lower left corner.
\begin{lemma}\label{relabelcol}
  The maps $f\colon\cTh_n(\R,Q_2(P,D))\mapsto\cTh_n(\R,Q_1(P,D))$ are well-defined in $P,D$ and computable uniformly in $n$. In fact, $Q_1(P,D)$ is obtained by postcomposing $Q_2(P,D)$ with a map
  \[(I\amalg I\times I)\amalg(I\amalg I\times I)\times(I\amalg I\times I)\to I\amalg I\times I.\]
\end{lemma}
\begin{proof}
  We abbreviate $Q_i=Q_i(P,D)$. Recall that $\R_E(S)$ maps $x$ to the ordered $1$- or $2$-tuple of labels $S(y)$ of elements $y$ of the fiber $\pi_E^{-1}(\set x)$.\par
  In particular, $Q_1\colon\R\to I\amalg I\times I$ maps $x$ to the ordered $1$- or $2$-tuple $(P(y))_{y\in\pi_\C^{-1}(\set x)}$ of labels $P(y)$ of elements $y$ of the fiber of $x$ under $\pi_\C\colon\C\to\R$. For instance, if $\pi_\C^{-1}(\set x)=\set{y_1,y_2}$ with $y_1<y_2$, then $Q_1(x)=(P(y_1),P(y_2))$.\par
  Similarly, $Q_2\colon\R\to(I\amalg I\times I)\amalg(I\amalg I\times I)\times(I\amalg I\times I)$ maps $x$ to the ordered tuple of ordered tuples $((P(z))_{z\in\pi_\C^{-1}(\set y)})_{y\in\pi_{\pi_\C(D)}^{-1}(\set x)}$. First, suppose that $\pi_\C^{-1}(\set y)\subseteq D$, or equivalently $\pi_\C^{-1}(\set y)=\pi_\C|_D^{-1}(\set y)$. If we equivalently regard $Q_2$ as a map
  \[\R\to(I\amalg I\times I)\amalg(I\amalg I\times I)\times(I\amalg I\times I)\cong I\amalg I\times I\amalg I\times I\amalg I\times I\times I\amalg I\times I\times I\amalg I\times I\times I\times I,\]
  then it maps such an $x$ to an ordered tuple $(P(z))_{z\in\pi_D^{-1}(\set x)}$ of labels of elements of the fiber of $x$ under $\pi_{\pi_\C(D)}\circ\pi_\C|_D=\pi_D$. For instance, if $\pi_D^{-1}(\set x)=\set{x_1,x_2}$ with $x_1<x_2$, then $Q_2(x)=(P(x_1),P(x_2))$. This already describes the required map of finite sets on the first three summands $I\amalg I\times I\amalg I\times I$. Namely, it is the identity on the first summand $I$ and identifies the second summand $I\times I$ with the third summand $I\times I$.\par
  If $\pi_\C^{-1}(\set y)\nsubseteq D$, then the fiber of $\pi_{\pi_\C(D)}\circ\pi_\C$ contains unnecessary labels. However, it is still possible to determine which levels arise from the fiber of $\pi_D$. Namely, since $D$ is a Cantor subset of $\C$, the left endpoints of jumps in $D$ are no right endpoints of jumps in $\C$ and vice versa. In other words, the fiber of $\pi_D=\pi_{\pi_\C(D)}\circ\pi_\C|_D$ consists of the first and the last element of the fiber of $\pi_{\pi_\C(D)}\circ\pi_\C$. Thus an element of a summand $I\times I\times I$ or $I\times I\times I\times I$ is mapped to the pair in $I\times I$ of its first and its last entry.
\end{proof}
\begin{lemma}\label{compcantsub}
  The maps
  \[\ucsub\colon\cTh_{n+1}(\C,P)\mapsto\setb{\cTh_n(D,P)}{D\subseteq\C\text{ an }n\text{-uniformly labeled Cantor subset}}\]
  are well-defined in $P$ and computable uniformly in $n$. In fact, $\ucsub(t)$ is the union of the image under $f$ of \cref{relabelcol} of the Cantor $n$-subtypes of $t$ with $\set{t|n}$.
\end{lemma}
\begin{proof}
  If $D$ is nonmeager, then by uniformity $\cTh_n(D,P)=\cTh_n(\C,P)$. For meager $D$, the type $\cTh_n(D,P)$ is $\cTh_n(\R,Q_1(P,D))$ in the above discussion. Since $D\mapsto\pi_\C(D)$ in \cref{projcantorbij} is a well-defined surjection, the Cantor subtypes of $\cTh_{n+1}(\C,P)$ are precisely the $\cTh_n(\R,Q_2(P,D))$ as $D$ varies.
\end{proof}
\begin{definition}
  A $0$-type $t_0$ is \textbf{consistent with} a set $T$ of $n$-types if all realized labels of elements of $T$ are realized labels of $t_0$.
\end{definition}
\begin{lemma}\label{compunifprec}
  Let $(\R,P)$ be an $(n+1)$-uniform sum of $T$ and let $S$ denote the set of Cantor $n$-subtypes of $\cTh_{n+1}(\R,P)$. Then $t\in S$ if and only if at least one of the following conditions holds:
  \begin{parts}
  \item $t\in\bigcup_{s\in T}\ucsub(s)$.
  \item $t$ is the $n$-type of an $n$-uniform sum of some $T'\subseteq\bigcup_{s\in T}\ucsub(s)$. Moreover, $t|0$ is Cantor-satisfiable and consistent with $T'$, has the same comeager label as $P$, and all components of realized labels of $t|0$ are realized by $(\R,P)$.\label{compunifprecb}
  \end{parts}
\end{lemma}
\begin{proof}
  Let $(C_j)_j$ be as in \cref{unifsumdisj}. We first show that every $t\in S$ satisfies (a) or (b).\par
  First, suppose that some $C_j$ is nonmeager in a realization of $t$. Thus $C_j$ has nonempty interior and, by uniformity, $C_j$ includes a realization of $t$. Since $C_j$ is compatible with $T$, (a) follows.\par
  Otherwise, let $D$ be a realization of $t$ such that all $C_j$ are meager in $D$. We show (b). Since $\bigcup_jC_j$ is meager in $D$, $D$ has the same comeager label as $\R$ and it remains to show that $(D,P)$ is a uniform sum of some $T'$. We set $T'$ as the set of $t'\in\bigcup_{s\in T}\ucsub(s)$ that are realized in $D$. By \cref{trivunif} we may assume $D$ to be $(n+1)$-uniformly labeled, whence each element of $T'$ is realized in every nonempty open subset. It remains to show that $D$ includes countably many closed nowhere dense $E_j$ compatible with $T'$ such that $D\setminus\bigcup_jE_j$ has a cocountable label. Set $E_j=\overline{C_j\cap D}$. Then $D\setminus\bigcup_jE_j\subseteq\R\setminus\bigcup_jC_j$ has a cocountable label. For showing that $E_j$ is compatible with $T'$, let $F$ be an $n$-uniformly labeled Cantor subset of $E_j$. Then $F$ is a subset of $C_j$, thus by compatibility of $C_j$ a subset of a realization of some $s\in T$. Since $\cTh_n(F,P)$ is realized in $\overline D\supseteq E_j$ and thus also in $D$, we conclude that $\cTh_n(F,P)\in T'$.\par
  Since all $t$ satisfying (a) clearly lie in $S$, it remains to show that all $t$ satisfying (b) lie in $S$. For all $t'\in T'$ or realized labels $c$ of $t|0$ or components $c'$ of realized labels of jumps, and nonempty basic open $U_m\subseteq\R$, choose realizations $E_{t'm}\subseteq U_m$ of $t'$ and points $x_{cm}\in U_m,y_{c'm}\in U_m$ of labels $c,c'$ such that all $E_{t'm},\set{x_{cm}},\set{y_{c'm}}$ are included in pairwise distinct $C_{j_{t'm}},C_{j_{cm}},C_{j_{c'm}}$. Set $B=\bigcup_mB_m$ with $B_m=\bigcup_{t'\in T'}\overline{E_{t'm}}\cup\bigcup_c\set{x_{cm}}$ and $D_{c'}=\setb{y_{c'm}}m$. Let $A$ be the union of $\bigcup_jC_j\setminus\pare{B\cup\bigcup_{c'}D_{c'}}$ and all countable labels of $\R\setminus\bigcup_jC_j$. Then $A$ is $F_\sigma$ as $\bigcup_jC_j\setminus\pare{B\cup\bigcup_{c'}D_{c'}}$ is the union of all $C_{j_{t'm}}\setminus\overline{E_{t'm}},C_{j_{cm}}\setminus\set{x_{cm}},C_{j_{c'm}}\setminus\set{y_{c'm}}$ and all $C_j$ including no $E_{t'm},\set{x_{cm}}$ or $\set{y_{c'm}}$. Let $C$ be the Cantor set obtained from \cref{findcantor} with these $A,B,D_{c'}$ and $E$ given by the realized labels of jumps. Then $C$ has the given labels of jumps by construction; since $C$ contains an $x_{cm}$ in each nonempty open subset, $C$ realizes all realized labels of non-jumps; since $\bigcup_jC_j\subseteq A\cup B$ is meager in $C$, it has the same comeager label; and since $C$ is disjoint from $A$ and meets $\bigcup_{c'}D_{c'}$ only in jumps, $C$ omits all omitted labels of non-jumps. Thus, $C$ realizes $t|0$. Furthermore, $C$ includes an $E_{t'm}$ in each nonempty open subset, and since $C$ is disjoint from $A$, the restriction $C\setminus\bigcup_{t'm}E_{t'm}$ has a cocountable label, while each $C\cap E_{t'm}$, being a subset of $E_{t'm}$, is compatible with $T'$. Thus, $(C,P)$ is a uniform sum of $T'$.
\end{proof}
\begin{proposition}\label{unifcomp}
  Let $(\R,P)$ be an $n$-uniform sum of $T$. Then the maps $(\cTh_0(\R,P),T)\mapsto\cTh_n(\R,P)$ are well-defined in $T$ and $n$-uniform sums $(\R,P)$ and computable uniformly in $n$.
\end{proposition}
\begin{proof}
  We proceed by recursion on $n$, and $n=0$ is vacuous. If $(\R,P)$ is an $(n+1)$-uniform sum of $T$, then $\bigcup_{t\in T}\ucsub(t)$ is computable by \cref{compcantsub}. Using \cref{compunifprec}, we may recursively compute the set of Cantor subtypes.
\end{proof}
Well-definedness in \cref{unifcomp} yields uniqueness in the following the definition.
\begin{definition}
  For a $0$-type $t_0$ and a set $T$ of $n$-types, the $n$-uniform sum of $T$ with underlying $0$-type $t_0$ is the unique $n$-type that restricts to $t_0$ and is realized by an $n$-uniform sum of $T$.
\end{definition}
However, the conditions of \cref{defunifsum} are not expressible in terms of $\cTh_n$. Thus, sometimes $(\R,P)$ is no $n$-uniform sum of $T$, but $\cTh_n(\R,P)$ is an $n$-uniform sum of $T$.
\begin{proposition}\label{unifsumsat}
  Let $T$ be a set of Cantor-satisfiable $n$-types and $t_0$ be a $0$-type consistent with $T$. Then the $n$-uniform sum of $T$ with $0$-type $t_0$ is satisfiable.
\end{proposition}
\begin{proof}
  The iso-uniform sum of realizations of the summands and the isomorphism classes of points of all realized labels is an $n$-uniform sum with the given $0$-type $t_0$.
\end{proof}

\subsection{Computing coarse types}
The aim of this subsection is to prove decidability for coarse types (\cref{coarsecomput}). For this purpose, we prove that each coarse type is realized by an iterated iso-uniform sum of types that have a single realized label (\cref{satisconstr}).
\begin{lemma}\label{contunifsub}
  If $(\C,P)$ is an $n$-uniform sum of $T$ and $T'\subseteq T$, then $\C$ includes a meager Cantor set with the same $0$-type that is an $n$-uniform sum of $T'$.
\end{lemma}
\begin{proof}
  Replacing it by a subinterval, by \cref{trivunif} we may assume $P$ to be $(n+1)$-uniform. \Cref{unifmax} yields that $(\C,P)$ is an $(n+1)$-uniform sum of some set $T''$ with $T\subseteq\setb{t|n}{t\in T''}$. Now \cref{compunifprecb} yields an $n$-uniform sum of $T'\subseteq\bigcup_{s\in T''}\ucsub(s)$ with $0$-type $\cTh_0(\C,P)$.
\end{proof}
\begin{corollary}\label{descucsub}
  The set $\ucsub(t)$ is the image of the Cantor subtypes of $t$ under $f$ of \cref{relabelcol}.
\end{corollary}
\begin{proof}
  By \cref{compcantsub} it suffices to show that $t|n$ lies in this image, that is, is realized by a meager Cantor set. Apply \cref{contunifsub} with $T'=T$.
\end{proof}
\begin{definition}\label{partordcoarse}
  For coarse $n$-types $s,t$, write $s\le t$ if $s$ and $t$ have the same comeager label, the realized labels of $s$ form a subset of the realized labels of $t$ and the Cantor subtypes of $s$ form a subset of the Cantor subtypes of $t$. An $(n+k)$-type $s$ is $k$-\textbf{minimal} if $s$ is least with respect to this order among the satisfiable $t$ above $s|n$, that is, with $t|n=s|n$. We never regard an $n$-type for $n<k$ as $k$-minimal.
\end{definition}
For $n>0$, in $s\le t$ the condition that the realized labels form a subset follows from the condition that the Cantor subtypes form a subset. By \cref{descucsub}, $s\le t$ implies $\ucsub(s)\subseteq\ucsub(t)$.
\begin{lemma}\label{unifleast}
  The $(n+1)$-uniform sum of a set $T$ of Cantor-satisfiable $(n+1)$-types with a consistent $0$-type is the least satisfiable $(n+1)$-type with the given comeager labels and whose Cantor subtypes include $\bigcup_{t\in T}\ucsub(t)$.
\end{lemma}
\begin{proof}
  The sum is satisfiable by \cref{unifsumsat}.\par
  Let $s$ be a uniform sum of $T$ and $t$ be another such satisfiable $(n+1)$-type. \Cref{unifmax} yields that $t$ is an $(n+1)$-uniform sum of a set containing realizations of all elements of $\bigcup_{t'\in T}\ucsub(t')$. Then $s\le t$ follows as the description in \cref{compunifprec} shows that the set of Cantor subtypes of a uniform sum of $T$ is $\subseteq$-isotone in $\bigcup_{t'\in T}\ucsub(t')$.
\end{proof}
\begin{lemma}\label{unifleast'}
  The $(n+1)$-uniform sum of a set $T$ of $1$-minimal Cantor-satisfiable $(n+1)$-types with a consistent $0$-type is the least satisfiable $(n+1)$-type with the given comeager label and whose Cantor subtypes include $\set{t|n\mid t\in T}$.
\end{lemma}
\begin{proof}
  This follows from \cref{unifleast} since for an $1$-minimal $(n+1)$-type $t$, an $(n+1)$-type $s$ includes $\ucsub(t)$ if and only if $s$ contains $t|n$. Indeed, if $s$ contains $t|n$, then each realization of $s$ includes an $(n+1)$-uniform realization $t'$ of $t|n$, and so includes $\ucsub(t')\supseteq\ucsub(t)$.
\end{proof}
\begin{definition}
  \begin{parts}
  \item An $n$-type \textbf{construction} of \textbf{rank} at most $m\in\N$ is defined recursively as a pair of a satisfiable $0$-type $t$ and a set $T=\set{t_1,\dots,t_k}$ of $n$-type constructions of rank less than $m$ (for $m=0$, this permits only $T=\emptyset$). We denote the construction by $\sum^tT$ or $\sum_i^tt_i$.
  \item The value of an $n$-type construction $\sum^tT$ is the $n$-type $s$ with $s|0=t$ that is a uniform sum of the values of the elements of $T$. For this definition we recursively require that the values of the elements of $T$ are Cantor-satisfiable and $1$-minimal and that $t$ is consistent with $T$. If $\sum^tT$ has value $t'$, we also say that $\sum^tT$ is a construction of $t'$.
  \item A type $t$ is \textbf{constructible} of rank at most $m$ if $t$ has a construction of rank at most $m$. The type $t$ is constructible if $t$ is constructible of some rank. The rank of $t$, denoted by $\rk t$, is the minimal rank of a construction of $t$.
  \end{parts}
\end{definition}
For example, each satisfiable $0$-type $t$ is constructible. Indeed, $\sum^t\emptyset$ is a construction of $t$.
\begin{definition}
  The \textbf{canonical realization} of a construction $\sum^tT$ is defined recursively as the iso-uniform sum of the canonical realizations of all $s\in T$ and the isomorphism classes of points of all realized labels of $t$ with the comeager label of $t$. A labeling $P$ of $\C$ is a canonical realization of a construction if $\R_\C(P)$ is a canonical realization of this construction.
\end{definition}
Different constructions of a type may lead to non-isomorphic canonical realizations.
\begin{proposition}\label{satisofconstr}
  Constructible types are satisfiable.
\end{proposition}
\begin{proof}
  This follows by induction on a fixed construction using \cref{unifsumsat}. Explicitly, the canonical realization is a realization.
\end{proof}
\begin{proposition}\label{pseudocount}
  For a label $c$ of a constructible $n$-type $t$, the following conditions are equivalent:
  \begin{tfae}
  \item $c$ is meager and if $n>0$, then $c$ is meager in all Cantor subtypes of $t$.
  \item $c$ is meager and if $n>0$, then, recursively on $n$, $c$ satisfies (ii) in all Cantor subtypes of $t$.
  \item $c$ is countable in some realization of $t$.
  \item If $n=0$, then $c$ is countable in the canonical realization of $\sum^t\emptyset$. If $n>0$, then $c$ is countable in the canonical realizations of all constructions.
  \end{tfae}
\end{proposition}
\begin{proof}
  If (i) fails, then for each realization, $c$ is comeager in some Cantor set, thus uncountable, and so (iii) fails. Since clearly (iv)$\implies$(iii), it suffices to show (i)$\implies$(ii)$\implies$(iv).\par
  We show $\lnot$(ii)$\implies\lnot$(i) by induction on $n$. The case $n=0$ holds by definition. If $c$ is not pseudo-countable in some Cantor subtype $s$, then by the induction hypothesis $s$ contains a Cantor subtype $s'$ with $c$ comeager in $s'$. Then there is an $(n-1)$-uniform realization of $s'$, which witnesses that $c$ is comeager in some Cantor subtype of $t$.\par
  For (ii)$\implies$(iv), we use induction on the given construction. Since $c$ is countable in all summands by (ii) (using $n>0$) and the induction hypothesis, $c$ is countable in all $C_i$ in the definition of iso-uniform sums. It is countable in $\R\setminus\bigcup_iC_i$ as $c$ is not the comeager label.
\end{proof}
A label is \textbf{pseudo-countable} if the conditions of \cref{pseudocount} hold. Assuming determinacy for $\mc B$, a pseudo-countable label is countable in all realizations by the perfect set property. This fails for $\mc B=\sigma(\mathbf\Sigma^1_1)$ in the constructible universe.
\begin{proposition}\label{cantorsatisconstr}
  A constructible type is Cantor-satisfiable if and only if there is a realized label of jumps and all labels of jumps are pseudo-countable.
\end{proposition}
\begin{proof}
  Necessity follows as all labels of jumps are countable in all realizations. For sufficiency consider a realization of $t$ in $\R$ with all labels of jumps countable. Replacing each point of such a label by two points of the corresponding labels yields a Cantor set realizing $t$.
\end{proof}
In particular, Cantor-satisfiability of a constructible $n$-type $t$ for $n>0$ only depends on $t|1$.
\begin{proposition}
  \begin{parts}
  \item All $(n+1)$-uniform sums of $2$-minimal Cantor-satisfiable constructible $(n+1)$-types are $1$-minimal.
  \item For every constructible $n$-type $t$, there is a unique $1$-minimal $(n+1)$-type $t'$ above $t$. Moreover, $t'$ is constructible, $t\mapsto t'$ is computable, and $t'$ is Cantor-satisfiable if and only if $t$ is Cantor-satisfiable.
  \item If a $\mc B$-labeling of $\R$ has a Cantor subset with constructible $n$-type $t$, then it has a Cantor subset realizing $t'$ of (b).\label{unifmincont}
  \end{parts}
\end{proposition}
\begin{proof}
  We use simultaneous induction on $n$.\par
  (c) for $n-1$ implies (a) for $n$ and (a) holds for $n=0$. Indeed, let $t_1,\dots,t_m$ be $2$-minimal $(n+1)$-types (note that $m=0$ if $n=0$), and let $s$ be a satisfiable $(n+1)$-type above the uniform sum of $t_1|n,\dots,t_m|n$. We need to show that $s$ includes the uniform sum of $t_1,\dots,t_m$. By \cref{unifleast'} it suffices to show that $s$ contains $t_1|n,\dots,t_m|n$. This follows from (c) as $s|n$, being the uniform sum of $t_1|n,\dots,t_m|n$, contains $t_1|(n-1),\dots,t_m|(n-1)$.\par
  (a) for $n$ implies (b) for $n$. Indeed, let $t$ be constructible, say an $n$-uniform sum of $1$-minimal Cantor-satisfiable constructible types $s_1,\dots,s_m$ of smaller rank. By induction on the rank, there are $2$-minimal Cantor-satisfiable constructible $s_1',\dots,s_m'$ above $s_1,\dots,s_m$. Their uniform sum $t'$ is $1$-minimal by (a), is constructible and lies above $t$. Uniqueness holds as $\le$ is a partial order on coarse types. By induction on the rank, $s_i'$ has the same pseudo-countable labels as $s_i$. Since a label in a uniform sum is pseudo-countable if and only if it is pseudo-countable in all summands and meager, $t'$ has the same pseudo-countable labels as $t$. The claim about Cantor-satisfiability follows from \cref{cantorsatisconstr}.\par
  (b) for $n$ implies (c) for $n$. We use induction on $\rk t$. Let $t$ be an $n$-uniform sum of $1$-minimal types $s_1,\dots,s_m$ of smaller rank and $P$ be a realization. Then $P$ contains $s_1|(n-1),\dots,s_m|(n-1)$. By the induction hypothesis on $n$, $P$ contains $s_1,\dots,s_m$ (note that this also works for $n=0$ as then $m=0$). By the induction hypothesis on the rank, $P$ contains the $2$-minimal $(n+1)$-types $s_1',\dots,s_m'$ above $s_1,\dots,s_m$. Replacing it by a subinterval, we may assume that $P$ is $(n+2)$-uniform. \Cref{isunifsum} yields that $P$ is an $(n+1)$-uniform sum of $s_1',\dots,s_m'$ and some other types and thus by \cref{contunifsub} includes an $(n+1)$-uniform sum $t'$ of $s_1',\dots,s_m'$. This $t'$ is the $1$-minimal $(n+1)$-type above $t$ as defined in the proof of (b).
\end{proof}
The proof shows that the $1$-minimal type $t'$ above $t$ has the syntactically same construction as $t$ (just regarded as $(n+1)$-type construction, not as $n$-type construction). In particular, they have equal canonical realizations.
\begin{proposition}\label{satisconstr}
  A coarse type is satisfiable if and only if it is constructible. In fact, each satisfiable type is the uniform sum of the $1$-minimal types above its Cantor subtypes.
\end{proposition}
\begin{proof}
  Sufficiency follows from \cref{satisofconstr}. For necessity, let $t$ be a satisfiable $(n+1)$-type. Then by \cref{unifleast'} the uniform sum $s$ of the $1$-minimal $(n+1)$-types above the Cantor $n$-subtypes of $t$ is the least satisfiable type with $s\ge t$. Thus $t=s$, which is constructible as the Cantor subtypes of $t$ are constructible by induction.
\end{proof}
\begin{corollary}\label{coarsecomput}
  The set of satisfiable coarse $n$-types is computable.
\end{corollary}

\section{Uniform refinements}\label{secdecproofb}
From now on, we assume that all members of $\mc B_{\overline\C}$ are determined. The aim of this section is to show the following result.
\begin{proposition}\label{computh}
  \begin{parts}
  \item The maps $\cTh_{m+1}(\R,P)\mapsto\urefincant^{k^-}_{k_n}(\cTh_m)(\R,P)$ are well-defined in $P$ and computable uniformly in $k$ and $m$ when $k^-$ is large enough in terms of $m$.\label{computh2}
  \item There is $p$, computable from $k,m$, such that the maps $\cTh_p(\R,P)\mapsto\urefin^{k^-}_{k_n}(\cTh_m)(\R,P)$ are well-defined in $P$ and computable uniformly in $k$ and $m$ when $k^-$ is large enough in terms of $m$.\label{computh1}
  \end{parts}
\end{proposition}
We first deduce \cref{mainres} from \cref{computh}. For (a), by \cref{shelahres} it suffices to compute $\uTh$, that is, compute recursively $\urefincant$ and $\urefin$. Since increasing $k^-$ does not change their value by \cref{somewhereunif}, we may apply \cref{computh}, using computability of $\cTh$ by \cref{coarsecomput}.\par
It remains to show (b). Since our construction does not depend on the Boolean algebra, all sufficiently stable Boolean algebras share the same theory. In fact, if $P$ is a labeling for the smaller Boolean algebra, then all refinements constructed in the definition of $E$ in the proof of \cref{shelahres} and iso-uniform sums required to canonically realize the $\cTh_p$, as well as the refinements constructed in \cref{secrefin}, can be chosen in the smaller Boolean algebra, and the labeling restricted to a Cantor subset lies in the smaller Boolean algebra.
\begin{remark}
  Since S1S is reducible to the Borel monadic theory of order, any decision procedure requires iterated exponential time \cite{s1srunningtime}. The reduction of $\Th$ to $\uTh$ in \cref{shelahres} and the computation of $\cTh$ evidently only require this running time. However, the map of \cref{computh1} sending $(k,m)$ to $p$ requires us to compute higher levels of $\cTh$, thus increasing the running time of our procedure (which remains primitive recursive).\par
  The author does not expect this increase to be necessary. It even seems plausible that deciding the fragment of the theory expressed by $\uTh$ is possible in less than iterated exponential time and only the procedure of \cref{shelahres} is computationally expensive.
\end{remark}

\subsection{Refined Cantor subsets}\label{secrefincant}
We first discuss a general approach to both parts of \cref{computh}. To obtain $t\in\urefincant_k^f(\cTh_n)(\R,P)$, respectively $t\in\urefin_k^f(\cTh_n)(\R,P)$, there are a few necessary conditions. First, $t$ must be satisfiable for $\urefin$ and Cantor-satisfiable for $\urefincant$. Second, observe that the map of coarsening types along $f$, namely $\cTh_n(\R,Q)\mapsto\cTh_n(\R,f\circ Q)$, is well-defined and computable by an easy recursion on the definition of $\cTh_n$. Then, coarsening the type $t$ along $f$ must yield $\cTh_{n-1}(\R,P)$ for $\urefin$ and a Cantor $(n-1)$-subtype for $\urefincant$.\par
One might try to show that these conditions suffice. While this succeeds for \cref{computh2}, as we argue now, it fails for \cref{computh1}, and we will later phrase the obstructions as winning strategies of Separator in separation games.
\begin{lemma}\label{contcanon}
  Let $S$ be a construction of a coarse type $t$. Every Cantor set realizing $t$ has a Cantor subset that is a canonical realization of $S$.
\end{lemma}
\begin{proof}
  We use induction on $S$. Write $S=\sum^{t|0}\mc S$. Let $(\C,P)$ be a realization of $S$.\par
  The complement of the comeager label is included in a meager $F_\sigma$ set $M$. Increasing $M$, by the induction hypothesis we may assume that for every nonempty open $U$, the restriction $M\cap U$ includes canonical realizations of all $S'\in\mc S$. By \cref{refincantordense}, write $M=\bigcup_jC_j$ with pairwise disjoint closed $C_j$ such that for each nonempty open $U$ and $S'\in\mc S$, some $C_j\cap U$ includes a canonical realization of $S'$.\par
  We now proceed as in \cref{compunifprec}. For all $S'\in\mc S$ or realized labels $c$ of $t|0$ or components $c'$ of realized labels of jumps, and nonempty basic open $U_m\subseteq\R$, choose canonical realizations $E_{S'm}\subseteq U_m$ of $S'$ and points $x_{cm}\in U_m,y_{c'm}\in U_m$ of labels $c,c'$ such that all $E_{S'm},\set{x_{cm}},\set{y_{c'm}}$ are included in pairwise distinct $C_{j_{S'm}},C_{j_{cm}},C_{j_{c'm}}$. Set $B=\bigcup_mB_m$ with $B_m=\bigcup_{S'\in\mc S}\overline{E_{S'm}}\cup\bigcup_c\set{x_{cm}}$ and $D_{c'}=\setb{y_{c'm}}m$. Let $A$ be the union of $\bigcup_jC_j\setminus\pare{B\cup\bigcup_{c'}D_{c'}}$ and all countable labels of $\R\setminus\bigcup_jC_j$. Then $A$ is $F_\sigma$ as $\bigcup_jC_j\setminus\pare{B\cup\bigcup_{c'}D_{c'}}$ is the union of all $C_{j_{S'm}}\setminus\overline{E_{S'm}},C_{j_{cm}}\setminus\set{x_{cm}},C_{j_{c'm}}\setminus\set{y_{c'm}}$ and all $C_j$ including no $E_{S'm},\set{x_{cm}}$ or $\set{y_{c'm}}$. Let $C$ be the Cantor set obtained from \cref{findcantor} with these $A,B,D_{c'}$ and $E$ given by the realized labels of jumps. Then $C$ includes some $E_{S'm},\set{x_{cm}},\set{y_{c'm}}$ in each nonempty open subset. Let $(D_i)_i$ be an enumeration of the $E_{S'm}$ and the remaining points of labels distinct from $c_0$. Then the $D_i$ witness that $(C,P)$ is an iso-uniform sum of the canonical realizations of all $S'\in\mc S$ and the isomorphism classes of points of the realized labels, hence a canonical realization of $S$.
\end{proof}
\begin{lemma}\label{refincanon}
  Let $(\R,Q)$ be an $n$-uniform $\mc B$-refinement of $(\R,P)$. There is a construction $S$ of $\cTh_n(\R,P)$ such that the canonical realization of $S$ has an $n$-uniform $\mc B$-refinement realizing $\cTh_n(\R,Q)$.
\end{lemma}
\begin{proof}
  Choose a construction $S'$ of $\cTh_n(\R,Q)$. Coarsening $S'$ yields a construction $S$ of $\cTh_n(\R,P)$. Then the canonical realization of $S'$ is a refinement of the canonical realization of $S$ and realizes $\cTh_n(\R,Q)$.
\end{proof}
\begin{proposition}\label{refinsub}
  Let $(\C,Q)$ be an $n$-uniform $\mc B$-refinement of $(\C,P)$ such that $\cTh_n(\C,P)$ is a Cantor subtype of an $(n+1)$-type $t$. Then every realization of $t$ has a Cantor subset that has a $\mc B$-refinement realizing $\cTh_n(\C,Q)$.
\end{proposition}
\begin{proof}
  By \cref{contcanon} it suffices to consider the canonical realization $(\C,P')$ of the construction of $t$ with summands all constructions of minimal types above Cantor subtypes. In particular, $(\C,P')$ contains a canonical realization of each construction of $\cTh_n(\C,P)$. Thus \cref{refincanon} finishes the proof.
\end{proof}
\Cref{refinsub} yields \cref{computh2} as indicated above. Namely, the possible uniform $\mc B$-refinements of Cantor subsets of $(\R,P)$ are the satisfiable coarse types whose coarsening is a Cantor subtype of $(\R,P)$.

\subsection{Separation games}\label{secsepgame}
We prove \cref{computh1} using games, by induction on the number of labels. In the game we describe below, Separator aims to show the existence of a partition into sets of low complexity in the difference hierarchy \cite[22.E]{kechrisdst} of $F_\sigma$ sets using fewer labels. Pathfinder aims to show that, locally on a Cantor set, all labels have high complexity in the difference hierarchy. We formalize this using the building blocks $\setb{y\in k^\N}{\limsup y=i}$ for Wadge complete sets $\setb{y\in k^\N}{\limsup y\equiv i\mod2}$ for the first $\omega$ levels of the difference hierarchy.
\begin{definition}
  Consider a labeling $P\colon X\to\set{0,\dots,m-1}$ of $X\subseteq\R$ and $k\in\N_{>0}$. We define the \textbf{separation game} of $(P,k)$. There are two players, Pathfinder and Separator.\par
  Alternatingly, Pathfinder plays either a digit in $\set{0,\dots,9}$ or the radix point, and Separator plays an element of $\set{0,\dots,k-1}$. A play determines $x\in\R$ whose decimal representation is coded by Pathfinder and $y\in k^\N$ played by Separator. Pathfinder wins the play if $x\in X$ and $P(x)\equiv\limsup y\mod m$.
\end{definition}
If $X\in\mc B_\R$ and $P$ is a $\mc B$-labeling, then the separation game is determined.
\begin{remark}
  The separation game for $m=2$ is essentially a Wadge game for the sets $P^{-1}(\set0)$ and $\setb{y\in k^\N}{\limsup y\equiv0\mod2}$. There are two differences:
  \begin{itemize}
  \item Instead of $x\in\N^\N$, Pathfinder produces $x\in\R$. This difference is insignificant: For our argument, after minor changes we could restrict play to $\R\setminus C\cong\N^\N$ for a countable subset $C$ of the comeager label.
  \item Separator is not allowed to play waiting moves. In other words, our game is a Lipschitz game, not a Wadge game. However, playing $0\in\set{0,\dots,k-1}$ is at least as good as a waiting move, whence the Lipschitz and the Wadge game are equivalent, or more formally, won by the same player.
  \end{itemize}
  Taking $m>2$ corresponds to generalizing Wadge degrees to partitions (\cite{wadgepart}; for a more general account, see \cite{wadgebqo}). This observation could also be used to obtain an alternative proof of the following \cref{separatorwins}.
\end{remark}
\begin{lemma}\label{separatorwins}
  The following conditions are equivalent for a labeling $P$ and $k\in\N$:
  \begin{tfae}
  \item Separator has a winning strategy in the separation game of $(P,k)$.
  \item There is a partition $X=\bigcup_{i=0}^{k-1}X_i$ such that each $\bigcup_{i=0}^jX_i$ is $F_\sigma$ in $X$ and no label in $P(X_i)$ is congruent to $i$ modulo $m$.
  \end{tfae}
\end{lemma}
\begin{proof}
  Suppose that Separator wins. Declare $x\in X_i$ for the minimal $i$ such that there is a play that is compatible with the strategy and consists of a representation of $x$ and some $y$ with $\limsup y=i$. Since the strategy is winning, no label in $P(X_i)$ is congruent to $i$ modulo $m$. Finally, $\bigcup_{i=0}^jX_i$ is the set of $x\in X$ such that for one of the at most two representations of $x$, if Pathfinder plays this representation, then the $y$ played by Separator satisfies the $F_\sigma$ condition $\limsup y\le j$. Thus $\bigcup_{i=0}^jX_i$ is $F_\sigma$.\par
  Conversely, consider a partition $X=\bigcup_{i=0}^{k-1}X_i$ as in (ii). For each $j\in\set{0,\dots,k-1}$, write $\bigcup_{i=0}^{j-1}X_i=\bigcup_{n\in\N}C_{jn}\cap X$ with $C_{jn}$ closed. Given a position $a$, let $p_j(a)\in\N\cup\set\infty$ be the minimal $n$ such that $a$ has an extension in $C_{jn}$. Then Separator plays the maximal $j$ such that $p_j$ was increased by the immediately preceding move of Pathfinder. Let $x\in\R$ and $y\in k^\N$ form a play compatible with this strategy. Set $j=\limsup y$. Since $p_j$ is increased infinitely often, $x\notin\bigcup_nC_{jn}\supseteq\bigcup_{i=0}^{j-1}X_i$. Since $p_{j+1}$ is increased only finitely often, $x$ has arbitrarily long prefixes with extensions in $C_{j+1,n}$, where $n=p_{j+1}(a)$ for almost all prefixes $a$ of $x$. Since $C_{j+1,n}$ is closed, $x\in C_{j+1,n}\subseteq(\R\setminus X)\cup\bigcup_{i=0}^jX_i$. Hence, $x\notin X$ or $x\in X_j$, whence Separator wins.
\end{proof}
In particular, $X_j$ is the intersection of the $F_\sigma$ set $\bigcup_{i=0}^jX_i$ and the $G_\delta$ set $\bigcup_{i=j}^{k-1}X_i$. Using \cref{defcountable}, for each $k$, the winner of the separation game is expressible in restricted monadic logic by a formula $\phi_k(P)$.
\begin{example}
  There are sufficiently stable Boolean algebras whose members have the Baire property and such that the separation game is not determined. For instance, let $X$ be an uncountable set without a perfect subset. Then a separation game for $P\colon\R\to\set{0,1}$ with $P_1=X$ is undetermined. Indeed, a winning strategy of Separator for some $k$ would imply that $X$ is a Boolean combination of $F_\sigma$ sets, while a winning strategy of Pathfinder for $k\ge2$ would yield a perfect subset of $X$. This is the same argument as for classical Wadge games \cite[Theorem II.C.2]{wadgethesis}.
\end{example}

\subsection{Separable types}
We now phrase consequences of winning strategies in terms of coarse types, which will be employed in \cref{winnercoarse}.
\begin{definition}
  \begin{parts}
  \item We define by recursion on $k$ if an element of $\fcTh_n(\set{0,\dots,m-1})$ is \textbf{separable} of \textbf{index} $k$. A type $t$ is separable of index at most $k$ whenever $k-1$ is not congruent to the comeager label modulo $m$ and $t$ is a uniform sum with summands that are separable of index at most $k$ (as base case, this permits only empty sets of summands). For $1$-types and $k=1$, we additionally demand that label $0$ is omitted.
  \item We define by recursion on $k$ if an element of $\fcTh_n(\set{0,\dots,m-1})$ is \textbf{anti-separable} of \textbf{index} $k\le n$. A type is anti-separable of index at least $k$ whenever $k-1$ is congruent to the comeager label modulo $m$ and some Cantor $(n-1)$-subtype is anti-separable of index $k-1$. For $k=1$, we demand that label $0$ is realized instead of the condition on $k-1$.
  \end{parts}
\end{definition}
Let $P\colon\R\to I$ be a $\mc B$-labeling.\par
If an $n$-type $t$ is separable of some index, then so is $t|\ell$ for all $\ell\le n$. Conversely, if $t|k$ is separable of index $k$, then so is $t$. An $n$-type $t$ is anti-separable of index $k$ if and only if $t|k$ is so. Thus, we call $(\R,P)$ (anti-)separable of index $k$ if $\cTh_k(\R,P)$ is (anti-)separable of index $k$.\par
Relating separable types with winning strategies is straightforward:
\begin{lemma}\label{separatorwinstype}
  If Separator wins the separation game of $(P,k)$, then $(\R,P)$ is separable of index at most $k$.
\end{lemma}
\begin{proof}
  \Cref{separatorwins} yields a partition into an $F_\sigma$ set $X_{<k-1}$ and a $G_\delta$ set $X_{k-1}$ such that no label in $P(X_{k-1})$ is congruent to $k-1$ and Separator wins the separation game of $(P|_{X<k-1},k-1)$. By induction on $k$, each Cantor subset of $X_{<k-1}$ is separable of index less than $k$. Since the label congruent to $k-1$ does not occur, each Cantor subset of $X_{k-1}$ is separable of index at most $k$, and the comeager label is not congruent to $k-1$.
\end{proof}
To relate anti-separable types with winning strategies, we iterate the following lemma.
\begin{lemma}\label{pathfinderwins}
  Let $i\in\set{0,\dots,m-1}$ be a label such that Pathfinder wins the separation game of $(P|_X,m\cdot k+i+1)$ and $n\in\N$. Then there are an $n$-uniformly labeled $C\subseteq X$ and a relatively comeager $G_\delta$ subset $Y$ of $C$ such that $C$ is an interval or Cantor set, $Y$ has label $i$, and Pathfinder wins the separation game of $(P|_{C\setminus Y},m\cdot k+i)$.
\end{lemma}
\begin{proof}
  For $m=1$, Pathfinder wins on every nonempty $X$ and the lemma is easy. Thus assume $m>1$. The winning strategy of Pathfinder corresponds to a continuous map $\sigma\colon(m\cdot k+i+1)^\N\to X$.\par
  Let $C$ be a subinterval of the image of $\sigma$ that is $n$-uniformly labeled by the labeling that expands $P$ by the subset $\sigma(\setb y{\limsup y<m\cdot k+i})$. Replacing a finite initial segment of $\sigma$, we may arrange that $C=\im(\sigma)$. Set
  \[Y=C\setminus\sigma(\setb y{\limsup y<m\cdot k+i}).\]
  Since the image of the $F_\sigma$ subset $\setb y{\limsup y<m\cdot k+i}$ of a compact space is $F_\sigma$, its complement $Y$ is $G_\delta$. Since $\sigma$ is winning, $Y\subseteq\sigma(\setb y{\limsup y=m\cdot k+i})$ has label $i$. Since $\sigma((m\cdot k+i)^\N)\subseteq C\setminus Y$, the restriction of $\sigma$ is a winning strategy for the separation game of $(P|_{C\setminus Y},m\cdot k+i)$.\par
  We show that $C$ is a closed interval or Cantor set. First, $C=\im(\sigma)$ is a nonempty compact subset of $\R$. We show that there is no isolated point $x$. Otherwise $\sigma^{-1}(\set x)$ is nonempty open, thus meets the dense sets $\setb y{\limsup y=\ell}$ for $\ell\in\set{0,1}$, which contradicts that $\sigma$ is winning and $m>1$.\par
  It remains to show that $Y$ is comeager. Since $Y$ is $G_\delta$, it suffices to show that $Y$ is dense. By $0$-uniformity, it suffices to show that $Y$ is nonempty. Suppose otherwise. On the one hand, $\sigma$ also witnesses that Pathfinder wins the separation game of $(P|_{\sigma((m\cdot k+i)^\N)},m\cdot k+i)$. On the other hand, \cref{separatorwins} with
  \[X_0=\emptyset\text{ and }X_\ell=\sigma(\setb y{\limsup y=\ell-1})\setminus\sigma(\setb y{\limsup y<\ell-1})\text{ for }\ell>0\]
  shows that Separator wins this game. This is absurd.
\end{proof}
\begin{lemma}\label{pathfinderwinstype}
  If Pathfinder wins the separation game of $(P,k)$, then $(\R,P)$ is anti-separable of index at most $k$.
\end{lemma}
\begin{proof}
  First, replace $\R$ by the set $C$ of \cref{pathfinderwins}. Since $Y$ is comeager of the label congruent to $k-1$, it suffices to obtain a subset that is anti-separable of index $k-1$. Pathfinder wins the separation game of $(P|_{C\setminus Y},k-1)$, thus also for some closed subset of $C\setminus Y$. We finish by induction on $k$.
\end{proof}
\begin{corollary}\label{winnercoarse}
  \begin{parts}
  \item Separator wins the separation game of $(P,k)$ if and only if $(\R,P)$ is separable of index at most $k$.
  \item Pathfinder wins the separation game of $(P,k)$ if and only if $(\R,P)$ is anti-separable of index at least $k$.
  \end{parts}
\end{corollary}
\begin{proof}
  \Cref{separatorwinstype,pathfinderwinstype} establish necessity. Since by induction no type is both separable and anti-separable of the same index, sufficiency follows from determinacy.
\end{proof}
In particular, the winner of the separation game can be computed from a coarse type. This enables the case distinction in the proof of \cref{computh1} below.

\subsection{Refinements}\label{secrefin}
We now show \cref{computh1} for refinements of a $\mc B$-labeling $P\colon\R\to J$ along a fixed $f\colon I\to J$ by induction on the number of labels $\#J$.
\subsubsection{Case 1: Separator wins the separation game of $(P,k)$ for some $k$}
Let $t\in\fcTh_n(\set{0,\dots,m-1})$. Let $k\in\N$ and let $i\in\set{0,\dots,m-1}$ be the label with $i\equiv k-1\mod m$.
\begin{definition}
  Let $t$ be separable of index $k$.
  \begin{parts}
  \item The \textbf{lower summands} of $t$ are the Cantor $(n-1)$-subtypes of $t$ that are separable of index less than $k$.
  \item The \textbf{upper part} of $t$ is the type in $\fcTh_n(\set{0,\dots,m-1}\setminus\set i)$ that is obtained from $t$ by removing all Cantor subtypes with realized label $i$.
  \end{parts}
\end{definition}
\begin{lemma}\label{omitsepar}
  Let $t$ be separable of index $k$. Each Cantor subtype of $t$ without realized label $i$ is a lower summand of $t$.
\end{lemma}
\begin{proof}
  Since no Cantor subset has comeager label $i$, this follows by recursion on a construction.
\end{proof}
\begin{lemma}\label{stepseparable}
  Let $t$ be separable of index $k$. If $(\R,P)$ realizes $t$, then there is a comeager $G_\delta$ set $X$ such that the following conditions hold:
  \begin{parts}
  \item The labeling that equals $P$ on $X$ and is constantly equal to the comeager label of $P$ outside $X$ realizes the upper part of $t$; and
  \item each $(n-1)$-uniform Cantor subset of $\R\setminus X$ realizes a lower summand of $t$ and each nonempty open subset of $\R$ includes Cantor subsets of $\R\setminus X$ realizing all lower summands of $t$.
  \end{parts}
\end{lemma}
\begin{proof}
  \Cref{separatorwins} yields a $G_\delta$ set $X$ such that label $i$ does not occur in $X$ and Separator wins the separation game of $(P|_{\R\setminus X},k-1)$. Hence all types of Cantor subsets of $\R\setminus X$ are lower summands. Since some Cantor subtype of index $k$ is not realized outside $X$, but each nonempty open subset includes a realization, the $G_\delta$ set $X$ is dense, thus comeager.\par
  Removing some Cantor sets from $X$, we may arrange that each lower summand is realized in every nonempty open subset by a Cantor subset of $\R\setminus X$, that is, (b) holds.\par
  Let $s$ be a Cantor subtype without realized label $i$. By \cref{omitsepar}, $s$ is realized in every nonempty open subset by a Cantor subset of $\R\setminus X$. Write $\R\setminus X$ as a countable disjoint union of closed sets $C_j$. For all $s$ and basic open $U_m$, choose realizations $E_{sm}\subseteq U_m$ of $s$ that are subsets of pairwise distinct $C_j$. Replacing $X$ by $X\cup\bigcup_{sm}E_{sm}$, which remains $G_\delta$ as each $C_j\cap\bigcup_{sm}E_{sm}$ is locally closed, we can arrange that all $s$ occur in $X$ in every nonempty open subset. This guarantees (a).
\end{proof}
Thus, $X$ and the Cantor subsets of $\R\setminus X$ can be refined independently.
\begin{corollary}\label{refineparts}
  Let $t$ be a separable $n$-type of index $k$. If $(\R,P)$ realizes $t$, then the possible refinements of $\cTh_n(\R,P)$ along $f$ are obtained in the following way. Take any possible refinement $r_0$ of the upper part and for any lower summand $s$ take a nonempty set $R_s$ of possible refinements. Then a possible refinement is obtained by adding to a uniform sum decomposition of $r_0$ all elements of $\bigcup_sR_s$ as summands.
\end{corollary}
\begin{proof}
  Let $X$ be as in \cref{stepseparable}. Write $\R\setminus X$ as a countable disjoint union of Cantor sets and points $C_j$. Since we will define the refinement below independently on each $C_j$, we may replace $P$ on each $C_j$ by a $P'$ with $\Th_\ell(C_j,P)=\Th_\ell(C_j,P')$ for $\ell$ large enough to express the possible refinements. Thus, decomposing them further, by \cref{shelahresdecomp,relcoarunif} we may assume that all $C_j$ are $n$-uniformly labeled.\par
  Each Cantor set $C_j$ realizes some lower summand of $t$. We modify the $C_j$ such that every nonempty open subset includes $C_j$'s realizing all lower summands $s$ of $t$. For all $s$ and basic open $U_m$, choose Cantor sets $E_{sm}\subseteq U_m$ realizing $s$ that are subsets of pairwise distinct $C_j$. Decompose each $C_j$ that includes an $E_{sm}$ into this $E_{sm}$ and countably many relatively open intervals and points. By uniformity of the $C_j$, all resulting Cantor sets realize lower summands.\par
  Finally, refine $X$ as indicated by $r_0$ and the $C_j$ as indicated by the $R_s$ such that each element of $R_s$ is realized in every nonempty open subset. Since $X$ and the $C_j$ witness a uniform sum decomposition before refining, they do so afterwards.
\end{proof}
This enables a computation of the set of possible refinements by recursion. We first describe base cases. The refinements of types in a single label, that is, in $\cTh_n(\tsingl)$, are described by \cref{coarsecomput}. Next, types that are separable of index $1$ can be reduced to types in fewer labels.\par
The set of possible refinements of the upper part can be computed by recursion on the number of labels. The sets of possible refinements of lower summands can be computed by recursion on the index $k$ of the separable type $t$. \Cref{refineparts} finishes the computation of the set of possible refinements.
\begin{remark}
  Instead of the partition of \cref{stepseparable}, we could have chosen other ways of distributing Cantor subsets between $X$ and $\R\setminus X$. Our induction requires to place Cantor subsets containing label $i$ outside $X$ and Cantor subsets that are not separable of index less than $k$ inside $X$. In our construction, Cantor subtypes that do not contain label $i$, and thus in particular are separable of index less than $k$ by \cref{omitsepar}, are realized both outside and inside $X$. We could instead demand that they are realized only inside $X$ unless they embed into other separable Cantor subsets. It is not clear to the author whether we could alternatively demand that they are realized only outside $X$.
\end{remark}
\subsubsection{Case 2: Pathfinder wins the separation game of $(P,k)$ for large enough $k$}
It suffices to show that in this situation every satisfiable refinement of $\cTh_n(\R,P)$ is realized by a refinement of $(\R,P)$, as discussed in the beginning of \cref{secrefincant}.\par
The following strengthening of anti-separability can be obtained at the cost of increasing the quantifier rank.
\begin{definition}
  A coarse $n$-type is \textbf{maximal} if it is maximal with respect to the partial order $\le$ of \cref{partordcoarse}.
\end{definition}
Equivalently, a $0$-type is maximal if all labels are realized, and an $(n+1)$-type is maximal if all satisfiable types of Cantor sets are realized.
\begin{lemma}\label{existsmax}
  Let $m,n\in\N$ and $i\in\set{0,\dots,m-1}$. If a type $t\in\fcTh_{m\cdot(n+1)+i+1}(\set{0,\dots,m-1})$ is anti-separable of index at least $m\cdot(n+1)+i+1$, then $t$ has a Cantor subtype $s$ of comeager label $i$ such that $s|n$ is maximal.
\end{lemma}
\begin{proof}
  We proceed by induction on $n$. Passing to a Cantor subtype, we may assume that $t$ is anti-separable of index exactly $m\cdot(n+1)+i+1$. In this case, by \cref{contunifsub} it suffices to show that $t|n$ is maximal.\par
  By anti-separability, the comeager label is $i$. For $n=0$, all labels are comeager in same Cantor subtype, thus realized, and maximality follows. Suppose $n>0$. For all labels $j\in\set{0,\dots,m-1}$, by the induction hypothesis for $t|(m\cdot n+j+1)$ there are Cantor subtypes $t_j$ of comeager label $j$ such that $t_j|(n-1)$ is maximal.
\end{proof}
We say that a type $t\in\fcTh_n(I)$ is \textbf{maximal in the labels} $I'\subseteq I$ if only labels in $I'$ are realized in $t$, and the thus obtained type in $\fcTh_n(I')$ is maximal.
\begin{lemma}\label{suffmaxtype}
  Let $(\R,Q)$ be an $n$-uniform $\mc B$-refinement of $(\R,P)$ along $f\colon I\to J$. Suppose that there is $I'\subseteq I$ with $f(I')=f(I)$ such that all $(n-1)$-types that are maximal in the labels $I'$ are Cantor subtypes of $\cTh_n(\R,Q)$. Then every realization of $\cTh_n(\R,P)$ has an $(n-1)$-uniform $\mc B$-refinement realizing $\cTh_{n-1}(\R,Q)$.
\end{lemma}
\begin{proof}
  Let $P'\colon\R\to J$ realize $\cTh_n(\R,P)$. In the next paragraph, we proceed as in the second paragraph of \cref{contcanon}.\par
  The complement of the comeager label under $P'$ is included in a meager $F_\sigma$ set $M$. Increasing $M$, by \cref{contcanon} we may assume that for every nonempty open $U$, the restriction $M\cap U$ includes for each construction $S$ of a Cantor subtype of $\cTh_n(\R,P)$ a canonical realization of $S$. By \cref{refincantordense}, write $M=\bigcup_jC_j$ with pairwise disjoint closed $C_j$ such that for each nonempty open $U$ and $S$, some $C_j\cap U$ includes a canonical realization of $S$.\par
  Using \cref{refincanon}, refine these canonical realizations to obtain every Cantor subtype of $\cTh_{n-1}(\R,Q)$. Refine $\R\setminus\bigcup_jC_j$ by the comeager label of $(\R,Q)$. It remains to refine the remaining $C_j$ in a compatible way. For this, let $g\colon f(I)\to I'$ be such that $f\circ g=\id$ and coarsen the remaining $C_j$ along $g$. The resulting refined Cantor sets are compatible with the set of Cantor subtypes of $\cTh_n(\R,Q)$ since $\cTh_n(\R,Q)$ contains all maximal $(n-1)$-types in its realized labels.
\end{proof}
The assumptions of \cref{suffmaxtype} are satisfied by the following lemma, whose assumptions follow from \cref{existsmax} for anti-separable types of sufficiently large index.
\begin{lemma}
  Suppose that a type $t'$ is a refinement of $t$ along a map $f\colon I\to J$. There is $k$, computable from $f$ and $n$, such that whenever $t$ has a maximal Cantor $k$-subtype, then there is $I'\subseteq I$ with $f(I')=J$ such that all $n$-types that are maximal in the labels $I'$ are Cantor subtypes of $t'$.
\end{lemma}
\begin{proof}
  We use induction on $\#I$. If $t'$ realizes a Cantor set omitting some $i\in I$ whose coarsening is a maximal $k_1$-type for $k_1$ large enough, then the induction hypothesis for $I\setminus\set i$ applies.\par
  Otherwise, we show that we may choose $I'=I$ for large enough $k$. By induction on $n$, there is $k_2$ such that every Cantor subtype whose coarsening has a maximal $k_2$-type contains a maximal $(n-1)$-type. Set $k=k_1+k_2$.\par
  We show that an arbitrary maximal $n$-type, say with comeager label $i$, is realized. Write $t'$ as a nested uniform sum of types with $k_2$-maximal coarsenings. Let $t_1'$ be the type that results from replacing each of these summands by its underlying $0$-type without label $i$. Let $t_1$ be the coarsening of $t_1'$. Since $t$ contains a $k$-maximal type, $t_1$ contains a $k_1$-maximal type. By choice of $k_1$, $t_1'$ contains label $i$. Thus some Cantor subtype of $t'$ has comeager label $i$ and a Cantor subtype $t_2'$ with $k_2$-maximal coarsening. Finally, $t_2'$ contains each $(n-1)$-maximal type, whence $t'$ contains the given maximal $n$-type.
\end{proof}

\printbibliography
\end{document}